\documentclass[11pt,twoside,reqno]{amsart}

\usepackage{amssymb}
\usepackage{xcolor}
\usepackage[final]{hyperref}

\hypersetup{unicode= false, colorlinks=true, linkcolor=blue,
anchorcolor=blue, citecolor=green, filecolor=red, menucolor=blue, urlcolor=blue}

\def\ti{\tilde}

\def\PW{\text{\rm PW}}

\def\to{\rightarrow}

\def\pa{\partial}

\def \nlhat {\overset{\ \curlywedge}}

\def\R{{\mathbb R}}

\def\Z{{\mathbb{Z}}}
\def\C{{\mathbb{C}}}

\def\RR{{\mathcal{R}}}

\def\DD{{\mathcal{D}}}

\def\MM{{\mathcal M}}

\def\SS{{\mathcal S}}

\def\EE{{\mathcal E}}

\def\e{\varepsilon}
\def\d{\delta}
\def\DD{\Delta}

\def\l{\lambda}
\def\g{\gamma}
\def\G{\Gamma}
\def\a{\alpha}
\def\b{\beta}
\def\s{\sigma}

\newcommand{\charf}{\raisebox{\depth}{\(\chi\)}}

\def\arctanh{\operatorname{arctanh}}
\def\uv{\underline{v}_{\d}}
\def\S{\text{\rm Sinc}}

\theoremstyle{plain}
\newtheorem{lemma}{Lemma}
\newtheorem{theorem}{Theorem}
\newtheorem{corollary}{Corollary}

\newtheorem{proposition}{Proposition}

\newtheorem{remark}{Remark}

\numberwithin{equation}{section}

\author{A.~Poltoratski}
\address{University of Wisconsin\\ Department of Mathematics\\ Van Vleck Hall\\
480 Lincoln Drive\\
Madison, WI  53706\\ USA }
\email{poltoratski@wisc.edu}
\thanks{The author was partially supported by
NSF Grant DMS-2244801.}

\begin{document}

\begin{abstract} We discuss estimates for the maximal operator associated with the non-linear Fourier transform of an $L^2$-function on the half-line. For potentials with bounded dyadic $L^1$ masses we  prove weak-type maximal and maximal fluctuation estimates on the sets where the spectral densities are bounded away from zero and three classical maximal functions of the spectral data are uniformly small; such sets exhaust almost all of the real line as the parameters of their definition relax.
\end{abstract}

\title{A maximal estimate for the non-linear Fourier transform}


\maketitle

\section{Introduction}

We study one of the basic models of scattering corresponding to the 'real' Dirac system on the right half-line $\R_+$,
\begin{equation} \Omega \dot X =z X - QX,
	\label{eqDS}\end{equation}
where $z\in\C$ is a spectral parameter,
$$ \Omega=\begin{pmatrix} 0 & 1 \\ -1 & 0 \end{pmatrix},\textrm{ and  }Q(t)=\begin{pmatrix} 0 & f(t) \\ f(t) & 0 \end{pmatrix} $$
for some real-valued locally summable function $f$. A slightly more general form of the system allows
for a locally summable functions $g$ and $-g$ on the main diagonal of $Q$. The function $f+ig$
is then called the potential of the system.
To simplify our exposition, we keep the potential real, although our methods will work similarly
for the general potential.

We will be most interested in  the scattering
problems corresponding to the case $f\in L^2(\R_+)$. For each value of the spectral parameter $z$ the unknown function $$X(t,z)=\begin{pmatrix} u(t,z) \\ v(t,z) \end{pmatrix}$$
is assumed to be differentiable on $\R_+$  with respect to the time variable $t$ and satisfy an initial condition  $X(0,z)=c\in \R^2$.

A special role will be played by solutions  satisfying the Neumann, $$X(0,z)=\begin{pmatrix} 1 \\ 0 \end{pmatrix},$$
and Dirichlet, $$X(0,z)=\begin{pmatrix} 0 \\ 1 \end{pmatrix},$$ initial conditions.
The matrix function $M$ whose columns are the Neumann and Dirichlet solutions, i.e., the matrix-function which solves
\eqref{eqDS} with the initial condition $$M(0,z)=\begin{pmatrix} 1 & 0 \\ 0 & 1 \end{pmatrix},$$ is called
the fundamental matrix, or the transfer matrix,  of the system.

If $$M(t,z)=\begin{pmatrix} A(t,z) & B(t,z) \\ C(t,z) & D(t,z) \end{pmatrix}$$ is the fundamental matrix then the Hermite-Biehler functions of the system are defined as
$$E(t,z)=A(t,z)-iC(t,z)\textrm{ and }\ti E(t,z)=B(t,z)-iD(t,z).$$

We denote by $\Pi$ the Poisson measure on $\R$, $d\Pi(x)=dx/(1+x^2)$, and call a measure $\mu$ on $\R$ Poisson-finite if
$$\int\frac{d|\mu|(x)}{1+x^2}<\infty.$$
The family of de Branges spaces $B(E(t,\cdot)),\ t\in\R_+,$ possesses a unique positive Poisson-finite measure $\mu$ on $\R$ such that the embedding $B(E(t,\cdot))\to L^2(\mu)$ is isometric for all $t\in\R_+$. Similarly, the family $B(\ti E(t,\cdot))$ gives rise to a unique measure $\ti\mu$. The measures $\mu$ and $\ti\mu$ are called the spectral measures for the Dirac system \eqref{eqDS} corresponding to the Neumann and Dirichlet initial conditions at $0$ correspondingly. See \cite{dBr}, \cite{R} or \cite{MIF1} on the basics of spectral theory for canonical systems and \cite{Den} or \cite{Ro} on the reduction of Dirac systems to the canonical case.

Let $w(x)$ be the density of the absolutely continuous part of $\mu$, $d\mu_{ac}=w(x)dx$, and let $\ti w$ be the density of the absolutely continuous part of $\ti\mu$. For $f\in L^2(\R_+)$ the spectral measures satisfy the Szeg\"o condition
$$\log |w|, \ \log |\ti w|\in L^1(\Pi),$$
see the paper by Denisov \cite{Den} for this and many related results. In particular, $w,\ti w \neq 0$ a.e. on $\R$.

For $f\in L^2(\R_+)$ the system \eqref{eqDS} defines, for each $t>0$, the scattering coefficients $a(t,z)$ and $b(t,z)$, entire in $z$ and satisfying $|a|^2-|b|^2\equiv1$ on $\R$ (see Section 2 for the definitions); the ratio
$$\nlhat f_t\,(s)=\frac{b(t,s)}{a(t,s)},\qquad s\in\R,$$
is the non-linear Fourier transform of the truncated potential $f\charf_{(0,t)}$. The map $f\mapsto\nlhat f_t$ shares the basic structural properties of the linear Fourier transform --- the modulation and scaling symmetries, analyticity on half-lines, an analogue of the Riemann--Lebesgue lemma --- with the Plancherel theorem replaced by the non-linear Parseval identity
$$\int_\R\log|a(t,s)|\,ds\ =\ \frac\pi2\int_0^t f^2(\tau)\,d\tau .$$
We recall these facts in Section 2; for a detailed introduction to non-linear Fourier analysis see \cite{TT}, and for the closely related theory of Krein systems see \cite{Den}.

A basic open problem of the theory asks whether $\nlhat f_t\,(s)$ converges as $t\to\infty$ for almost every $s\in\R$ --- the non-linear analogue of Carleson's theorem \cite{C} on the almost everywhere convergence of Fourier series and integrals. The problem was raised by Muscalu, Tao and Thiele, who proved the corresponding statement for a Cantor group model of the transform \cite{MTT} and pointed out that, as in the linear case, the convergence should follow from a weak-type estimate for the associated maximal operator. For potentials in $L^p(\R_+)$, $1\leq p<2$, convergence follows from the work of Christ and Kiselev \cite{CK,CK1}, whose methods, however, do not extend to $p=2$ considered in this note. 

In \cite{Scatter} the pointwise behavior of the scattering data was studied through the motion of {\it resonances} --- the zeros of the Hermite--Biehler functions of the system in boxes of the natural spectral resolution $1/t$ --- and the a.e. convergence of the data was established along suitable sequences of times. The estimates of \cite{Scatter} are qualitative: the universality theorems underlying them carry no modulus of convergence. The present paper develops a quantitative version of that approach and proves a maximal estimate.

We consider potentials satisfying the {\it dyadic mass restriction}
$$M\ :=\ \sup_{n\geq0}\int_{2^n}^{2^{n+1}}|f(t)|\,dt\ <\ \infty,$$
a condition on the behavior of $f$ near infinity only (Section \ref{secMax}), and write $M'=\max\left(M,\int_0^1|f|\right)$. Let $w$ and $\ti w$ denote the densities of the absolutely continuous parts of the Neumann and Dirichlet spectral measures of the system, defined below, and let
$$A(s)\ =\ \log\left(\frac12\sqrt{\frac1{w(s)}+\frac1{\ti w(s)}+2}\right).$$
Our main result, Theorem \ref{thmMaxR} in Section \ref{secMax}, is a weak-type maximal fluctuation estimate. For every $\eta_0\in(0,\tfrac12]$ and $\lambda>0$ we define a family of {\it regular sets} $\RR(T,\eta_0,\lambda)\subset\{\min(w,\ti w)\geq\eta_0\}$, $T\geq1$, through the smallness of three classical maximal functions of the spectral data; the sets increase in $T$ and, for every fixed $(\eta_0,\lambda)$, exhaust almost all of $\{\min(w,\ti w)\geq\eta_0\}$ as $T\to\infty$. The theorem states that
$$\begin{gathered}\left|\left\{s\in\RR(T,\eta_0,\lambda):\ \sup_{t\geq T}\big|\log|a(t,s)|-A(s)\big|>\lambda\right\}\right|\ \leq\\ \leq\ K\,(M'+1)\left(\eta_0\min(\lambda,1)\right)^{-N}\,\frac{\|f\|^2_{L^2((T/32,\infty))}}{\lambda}\end{gathered}$$
with absolute constants $K$ and $N$. As $T\to\infty$ the right-hand side tends to zero, and we obtain the almost everywhere convergence
$$\log|a(t,s)|\ \longrightarrow\ A(s)\qquad(t\to\infty)$$
(Corollary \ref{corConv}); through the identity $|\nlhat f_t|^2=1-|a|^{-2}$ this identifies the a.e. limit of $|\nlhat f_t\,|$ --- the modulus layer of the convergence problem. The restriction to the modulus seems to be essential: the phases $\arg a(t,s)$ may diverge everywhere, while the ratio $\nlhat f_t=a/b$ converges, see\cite{Den2}.


The paper is organized as follows. Section 2 collects the preliminaries on Dirac systems and their scattering data. Section \ref{secU} introduces the maximal functions, the universality conditions and the regular sets. Section \ref{secRep} contains the representations of the Hermite--Biehler functions at the resonances, and Section \ref{secPMT} the scattering coefficients of a time interval. Section \ref{secTE} is devoted to the two-endpoint estimate, and Section \ref{secMax} to the maximal estimates and the convergence corollary. Appendix \ref{secApp} derives the kernel approximation (U2) from the maximal conditions.

\section{Preliminaries}

Here and throughout the paper, for an entire function $F$ we denote by $F^\#$ the entire function
$$F^\#(z)=\overline{F(\bar z)}.$$

The Hermite-Biehler functions of the system satisfy the differential equation
\begin{equation} \frac \pa {\pa  t} E(t,z)=-izE(t,z) + f (t)E^\# (t,z),\label{eqME}\end{equation}
and $\ti E$ satisfies the same equation. Since the entries $A,B,C,D$ of the fundamental matrix are real entire functions and $E=A-iC$, $\ti E=B-iD$, we have $E^\#=A+iC$ and $\ti E^\#=B+iD$.

It follows from \eqref{eqDS} that
\begin{equation}\det M(t,z)= 1\label{det=1}\end{equation}
for all $t$ and $z$.
Rewritten in terms of $E$ and $\ti E$, this relation becomes
\begin{equation}\det\begin{pmatrix} E &  \ti E \\ E^\# & \ti E^\#
	\end{pmatrix}\equiv 2i.\label{eqDet2i}\end{equation}

On the real line, where $E^\#=\bar E$ and $\ti E^\#=\bar{\ti E}$, \eqref{eqDet2i} becomes
\begin{equation}\Im \left(E\bar{\ti E}\right)\equiv 1\ \text{ on }\R.\label{eqImE}\end{equation}

The reproducing kernels of the Paley-Wiener spaces are constant multiples of sinc functions:
$$ \S(t,w,z)=\frac1\pi\frac{\sin t(z-\bar w)}{z-\bar w},$$
and for the $L^2$-norms we have
$$ ||\S(t,w,\cdot)||_2^2= \S(t,w,w)=\frac {\sinh 2ty}{2\pi y},\ y=\Im w.$$

Further, for each $t\geq 0$ define entire functions $a(t,z)$ and $b(t,z)$ as
\begin{equation}\begin{gathered}a(t,z)=\frac {\EE(t,z) +i\ti\EE(t,z)}2=\frac {e^{itz}}2 (E(t,z)+i\ti E(t,z)), \\  b(t,z)=
		\frac {\EE(t,z) -i\ti\EE(t,z)}2=\frac {e^{itz}}2 (E(t,z)-i\ti E(t,z)),\label{eqab}\end{gathered}\end{equation}
where $\EE(t,z)=e^{itz}E(t,z)$ and $\ti\EE(t,z)=e^{itz}\ti E(t,z)$.
(Note that our notations are slightly different from those in \cite{TT} where $a$ stands for $a^\#$ in our definitions.)

Together with \eqref{eqImE}, the definitions \eqref{eqab} imply
$$|a|^2-|b|^2\equiv 1\ \text{ on }\R.$$
In particular, $|a|\geq 1$ on $\R$. For $t>0$ we denote by $f_t=f\charf_{(0,t)}$ the truncation of the potential and define its non-linear Fourier transform as
$$\nlhat f_t\,(s)=\frac{b(t,s)}{a(t,s)},\qquad s\in\R;$$
by the last identity, $|\nlhat f_t|<1$ on $\R$.

We will also use the following exact consequence of \eqref{eqab} and \eqref{eqImE}, the {\it elevation identity}: for all $t>0$ and $s\in\R$,
\begin{equation}
|a(t,s)|^2=\frac{|E(t,s)|^2+|\ti E(t,s)|^2+2}4 .
\label{eqAid}
\end{equation}
Indeed, $4|a|^2=|E+i\ti E|^2=|E|^2+|\ti E|^2+2\Im\left(E\bar{\ti E}\right)$ on $\R$, and the last term equals $2$ by \eqref{eqImE}.

We will denote by $\PW_t$, $t>0$, the Paley-Wiener space, i.e. the space of all entire functions of exponential type at most $t$ whose restrictions to $\R$ belong to $L^2(\R)$; equivalently, $\PW_t$ is the space of Fourier transforms of functions from $L^2(-t,t)$. An entire function $F$ belongs to $\PW_t$ iff $e^{itz}F,\ e^{itz}F^\#$ belong to $H^2(\C_+)$.

Recall that an entire function $F$ is called a Hermite-Biehler (HB) function if
$$|F^\#(z)|<|F(z)|\qquad\text{ for all }z\in\C_+;$$
in particular, an HB function has no zeros in $\C_+$. For an HB function $F$, the de Branges space $B(F)$ is the Hilbert space of all entire functions $H$ such that both $H/F$ and $H^\#/F$ belong to $H^2(\C_+)$, with the norm $||H||_{B(F)}=||H/F||_{L^2(\R)}$. See \cite{dBr} for these notions. The functions $E(t,\cdot)$ and $\ti E(t,\cdot)$ introduced above are HB functions for every $t\in\R_+$.

Throughout the paper, $\mu$ denotes the spectral measure of the Dirac system \eqref{eqDS} corresponding to the Neumann initial condition at $0$, as defined in the introduction, and $\ti\mu$ denotes the spectral measure corresponding to the Dirichlet initial condition. For a Poisson-finite measure $\nu$ on $\R$ we write
$$||g||_\nu=\left(\int_\R |g|^2d\nu\right)^{1/2}\qquad\text{ and }\qquad <g,h>_\nu=\int_\R g\bar h\, d\nu$$
for the norm and the inner product in $L^2(\nu)$; in particular, $||\cdot||_\mu$ and $<\cdot,\cdot>_\mu$ stand for the norm and the inner product in $L^2(\mu)$.

Recall that the embedding $B(E(t,\cdot))\to L^2(\mu)$ is isometric; for our system, $B(E(t,\cdot))$ coincides with $\PW_t$ as a set (see \cite{Den}). We denote by $K(t,\l,z)$ the reproducing kernel of $B(E(t,\cdot))$ with respect to the inner product of $L^2(\mu)$: for every $F\in B(E(t,\cdot))$ and every $\l\in\C$,
$$F(\l)=<F,K(t,\l,\cdot)>_\mu.$$
In particular,
$$K(t,\l,\l)=||K(t,\l,\cdot)||^2_\mu=\max\{|F(\l)|^2\ |\ F\in\PW_t,\ ||F||_\mu\leq 1\}.$$
The kernels $\ti K(t,\l,z)$ of the Dirichlet chain $B(\ti E(t,\cdot))\to L^2(\ti\mu)$ are defined in the same way.

We denote by ${\mathcal D}(t,\l,z)$ and $\ti{\mathcal D}(t,\l,z)$ the numerators of the kernels:
\begin{equation}
	{\mathcal D}(t,\l,z)=\det\begin{pmatrix}A(t,z) & \bar A(t,\l) \\ C(t,z) & \bar C (t,\l)
	\end{pmatrix}=\frac 1{2i}\det\begin{pmatrix}E(t,z) & E(t,\bar \l) \\ E^\#(t,z) & E^\# (t,\bar \l)
	\end{pmatrix},\label{eq1004M}\end{equation}
and $\ti{\mathcal D}$ is defined by the same formulas with $B,D,\ti E$ in place of $A,C,E$. By the standard formula for the reproducing kernels of de Branges spaces (see \cite{dBr}),
\begin{equation}
K(t,\l,z)=\frac{{\mathcal D}(t,\l,z)}{\pi(\bar\l-z)},\qquad
\ti K(t,\l,z)=\frac{\ti{\mathcal D}(t,\l,z)}{\pi(\bar\l-z)} .
\label{eqKD}
\end{equation}

Let us denote by $\g(p)$ the function $$\g(p)=\sqrt{2}/\sqrt{\sinh [2p]}$$
defined for $p> 0$. For $s\in\R$ and $a>0$ we denote by $Q(s,a)$ the box
$$Q(s,a)=\{z\ |\ |\Im z|<a,\ |\Re z-s|<a\}$$
centered at $s$. If $\zeta=x-iy$, $y>0$, is a zero of $E(t,\cdot)$ (all zeros of $E(t,\cdot)$ lie in $\C_-$), we call $v=ty$ its {\it depth}; note that a zero lying in $Q(s,C/t)$ automatically has depth less than $C$.

\section{Maximal functions, universality conditions and regular sets}\label{secU}

The regular sets, on which the maximal estimates of Section \ref{secMax} will be proved, are defined in this section through the smallness of three classical maximal functions of the spectral data. The proofs themselves use the spectral universality conditions (U2)--(U3) stated below; the point of defining the regular sets through the maximal functions is that the universality conditions are hard to control directly, while maximal functions are classical objects amenable to direct estimates, and the former follow from the latter (Theorem \ref{thmAppMain}, proved in Appendix \ref{secApp}, and Proposition \ref{propU3} below).

{\bf Maximal functions.} We write
$$P_z(x)=\frac1\pi\,\frac{\Im z}{(x-\Re z)^2+(\Im z)^2},\qquad P\nu(z)=\int_\R P_z\,d\nu$$
for the Poisson kernel and the Poisson extension of a Poisson-finite measure $\nu$, and $\G_\s=\{z\in\C_+:\ |\Re z-\s|<\Im z\}$ for the cone over $\s\in\R$. For $h>0$ define the truncated maximal functions
$$M_h\nu(\s)=\sup\left\{P|\nu|(z):\ z\in\G_\s,\ \Im z\leq h\right\},$$
$$\MM_h g(\s)=M_h\!\left(|g-g(\s)|\,dm\right),\qquad
\MM^2_h G(\s)=M_h\!\left(|G-G(\s)|^2dm\right),$$
for a function $g$ with $|g-g(\s)|\in L^1(\Pi)$ and for $G\in L^2(\Pi)$. We denote by $G$ and $\ti G$ the outer functions in $\C_+$ with $|G|^2=w$, $|\ti G|^2=\ti w$ a.e.\ on $\R$; they exist by the Szeg\"o condition of Section 2, and belong to $L^2(\Pi)$ since $\mu$, $\ti\mu$ are Poisson-finite.

{\it Conditions (M).} Let $C\geq1$, $t\geq1$ and $\e_0\in(0,\tfrac14]$. We say that $(s,t)$ satisfies {\rm(M1)--(M3)} with constant $\e_0$ at scale $C$ if
$$\begin{gathered}
\text{(M1)}\qquad \MM_{50C/t}\!\log w\,(s)\ \leq\ \e_0,\\
\text{(M2)}\qquad \MM^2_{50C/t}G(s)\ \leq\ \e_0\,w(s),\\
\text{(M3)}\qquad M_{50C/t}\,\mu(s)\ \leq\ (1+\e_0)\,w(s),
\end{gathered}$$
and the same inequalities hold for the Dirichlet data $(\ti w,\ti G,\ti\mu)$ with $\ti w(s)$ in place of $w(s)$. We refer to the parameter $C$ as the {\it scale}: at time $t$ the natural resolution of the spectral data near $s$ is $1/t$, and all conditions of this section --- the truncations above and the boxes $Q(s,\cdot/t)$ below --- live at distances from $s$ that are bounded multiples of $C/t$.

{\bf Universality conditions.} The arguments of Sections \ref{secRep}--\ref{secMax} use the universality of the spectral data near almost every point of $\R$, in the form of the following two conditions on a pair $(s,t)$, $s\in\R$, $t>0$, depending on the parameters $C\geq 1$ and $\e\in(0,\tfrac14]$. Below $w=w(s)$, $\ti w=\ti w(s)$.

{\it (U2) (kernel approximation)} 
$$\begin{gathered}
\sup_{\l,z\in Q(s,16C/t)}\left|{\mathcal D}(t,\l,z)-\frac{\sin[t(\bar\l-z)]}{w}\right|\leq\e
\qquad\text{ and }\\
\sup_{\l,z\in Q(s,16C/t)}\left|\ti{\mathcal D}(t,\l,z)-\frac{\sin[t(\bar\l-z)]}{\ti w}\right|\leq\e .
\end{gathered}$$

{\it (U3) (exponential models at resonance-free times)} For each of the values $c=C$, $2C$, $4C$, $8C$ the following holds: if $E(t,\cdot)$ has no zeros in $Q(s,c/t)$, then there is a unimodular constant $\a$ such that
$$\left|E(t,x)-\frac{\a}{\sqrt w}\,e^{-itx}\right|\leq\e\qquad\text{ for all }x\in\R\cap Q(s,c/t),$$
and the same holds for $\ti E$ (with its own unimodular constant $\ti\a$) whenever $\ti E(t,\cdot)$ has no zeros in $Q(s,c/t)$.

{\bf From (M) to (U).} By Theorem \ref{thmAppMain}, proved in Appendix \ref{secApp}, conditions (M1)--(M3) with a small enough constant imply (U2) with any prescribed accuracy, with explicit constants; by Proposition \ref{propU3} below, condition (U3) then follows from (U2) taken on an enlarged box, with the accuracy limited only by the depth up to which the enlarged box is free of resonances. This makes it possible to define the regular sets through the maximal conditions alone. The $o(1)$-forms of both conditions are in \cite{Scatter}: (U2) is Lemma 5 and Corollary 1 there, and (U3) is Corollary 4, part 2). Earlier treatments also used the norm universality $wK(t,z,z)/\S(t,z,z)=1+o(1)$ on the full box (Lemma 2 of \cite{Scatter}), whose known proof at the points near the real axis is a compactness argument without an explicit modulus; the present paper avoids that condition altogether: the normalization of the Hermite--Biehler functions at the resonances is derived from (U2), in an additive form at all depths and in the multiplicative form above a thin shallow layer (Lemma \ref{lemBeta}), and the layer itself is excluded, outside an exceptional set whose measure is controlled by the tail of the potential, by the first-entry estimate of Lemma \ref{lemShallow}.

We also recall that $\sqrt{w\ti w}\leq1$ a.e.\ on $\R$: this follows from the Alexandrov--Clark representation of the pair $(w,\ti w)$ through a single contractive analytic function (see \cite{MIF1}),
$$w=\frac{1-|\phi|^2}{|1-\phi|^2},\qquad \ti w=\frac{1-|\phi|^2}{|1+\phi|^2},\qquad
\sqrt{w\ti w}=\frac{1-|\phi|^2}{|1-\phi^2|}\leq1 .$$
Consequently, on the set $\{\min(w,\ti w)\geq\eta_0\}$,
\begin{equation}
\eta_0\ \leq\ \min(w,\ti w)\ \leq\ \max(w,\ti w)\ \leq\ \frac1{\eta_0},
\qquad
\frac 1{\sqrt{w}}+\frac 1{\sqrt{\ti w}}\ \geq\ \frac 2{(w\ti w)^{1/4}}\ \geq\ 2 .
\label{eqxy2}
\end{equation}

{\bf Regular sets.} For $\eta_0\in(0,\tfrac12]$, $\d\in(0,\tfrac14]$, $C\geq1$ and $T\geq1$ put
\begin{equation}
C^\sharp=C^\sharp(C,\d):=C+K_9\left(1+\log(1/\d)\right),
\qquad
\e_0(\d,C,\eta_0)=\d^{K_{10}}\,\eta_0^{8}\,e^{-K_8C},
\label{eqEpsZero}
\end{equation}
with the absolute constant $K_8$ of Theorem \ref{thmAppMain} and absolute constants $K_9$, $K_{10}=8+K_8K_9$, and define the {\it regular set} $\RR(T,\eta_0,C,\d)$ as the set of all $s\in\R$ with $\min\left(w(s),\ti w(s)\right)\geq\eta_0$ such that

{\it for all $t\geq T/32$, conditions (M1)--(M3) hold at $(s,t)$ with the constant $\e_0(\d,C,\eta_0)$ at scale $C^\sharp$.}

(The factor $\tfrac1{32}$ allows the arguments of Section \ref{secMax} to look back several octaves from a time $t\geq T$; it plays no other role.)

The definition involves the maximal functions only. By Theorem \ref{thmAppMain}, applied at the scale $C^\sharp$ with $\e=\d$ (note that $\min(w,\ti w,1)\geq\eta_0$ by \eqref{eqxy2}, and that $\e_0\leq\d^8\eta_0^8e^{-K_8C^\sharp}$ by the choice of $K_{10}$), we obtain:
\begin{equation}
\begin{gathered}
\text{\it for every }s\in\RR(T,\eta_0,C,\d)\text{\it , condition (U2) holds at $(s,t)$, at the scale $C^\sharp$}\\
\text{\it (hence also at the scale $C$), with $\e=\d$, for all }t\geq T/32 ,
\end{gathered}
\label{eqRtoU}
\end{equation}
and condition (U3) at the scale $C$ is recovered from \eqref{eqRtoU} by the following proposition, proved at the end of the section. As in \cite{Scatter}, for $c>0$ we denote by $T_0(s,c)$ the set of all $t>0$ for which $E(t,\cdot)$ has zeros in the box $Q(s,c/t)$, and by $\ti T_0(s,c)$ the analogous set for $\ti E$; thus $t\notin T_0(s,c)$ means that the box $Q(s,c/t)$ contains no zeros of $E(t,\cdot)$.

\begin{proposition}\label{propU3}
There are absolute constants $K_9$, $K_{11}$ such that the following holds on $\RR(T,\eta_0,C,\d)$ for $t\geq T$ and $c\in\{C,2C,4C,8C\}$, provided that $C\geq4\pi$ and $4\d\leq\eta_0$. (Both conditions hold in all applications below: there $C\geq C_1(\eta_0)\geq4\pi$, with $C_1$ of Lemma \ref{lemPair}, and the thresholds imposed on $\d$ are always taken smaller than $\eta_0/4$.)

1) {\rm(lattice rigidity)} If $t\notin T_0(s,c)$, then every zero of $E(t,\cdot)$ in $Q(s,8C^\sharp/t)$ has depth at least $c-1$.

2) {\rm(resonance-free models)} If $t\notin T_0(s,c)$ and $v_{\min}\in[c-1,\infty]$ denotes the smallest depth of the zeros of $E(t,\cdot)$ in $Q(s,8C^\sharp/t)$, then there is a unimodular constant $\a$ such that
$$\left|E(t,x)-\frac{\a}{\sqrt w}\,e^{-itx}\right|\ \leq\ K_{11}\left(\d+e^{-2v_{\min}}\right)\eta_0^{-1}
\qquad\text{ for all }x\in\R\cap Q(s,c/t).$$
In particular, if no zeros of $E(t,\cdot)$ in $Q(s,8C^\sharp/t)$ have depth $\leq(C^\sharp-C)/2$, then (U3) holds at $(s,t)$ on the box $Q(s,c/t)$ with $\e=K_{11}\,\d\,\eta_0^{-1}$. The same statements hold for $\ti E$.
\end{proposition}

The maximal functions are Borel functions of $s$; hence the sets $\RR$ are measurable. They increase in $T$, and, by the almost-everywhere-eventually property (Proposition \ref{propAae}, proved in Appendix \ref{secApp}), for every fixed $\eta_0,C,\d$ almost every point of $\{\min(w,\ti w)\geq\eta_0\}$ belongs to $\RR(T,\eta_0,C,\d)$ for some $T$. The union of the sets $\{\min(w,\ti w)\geq\eta_0\}$ over $\eta_0>0$ has full measure.

\begin{proof}[Proof of Proposition \ref{propU3}]
Throughout, (U2) is used at the scale $C^\sharp$, i.e.\ on the box $Q(s,16C^\sharp/t)$, with $\e=\d$, as provided by \eqref{eqRtoU}.

1) Suppose $\zeta$ is a zero of $E(t,\cdot)$ in $Q(s,8C^\sharp/t)$ of depth $v\leq c-1$, and put $\b=A(t,\zeta)$ as in \eqref{eqDslot}. Then $\b\neq0$: otherwise ${\mathcal D}(t,\zeta,\cdot)\equiv0$ by \eqref{eqDslot}, while (U2) forces $|{\mathcal D}(t,\zeta,z)|\geq|\sin[t(\bar\zeta-z)]|/w-\d\geq1/(2w)$ at suitable $z$. By the conjugated slot identity (the display \eqref{eqMappE} below, whose derivation uses only (U2) and \eqref{eqDslot} and is valid on the present box), $|i\b E(t,z)-\sin[t(z-\zeta)]/w|\leq\d$ there. The function $\sin[t(z-\zeta)]$ vanishes exactly at the lattice points $z_k=\zeta+k\pi/t$, $k\in\Z$, all of depth $v$, and on the circle $|z-z_k|=r/t$, $r\leq1$, its modulus is at least $r/2$. Taking $r=4\d w$ (so that $r\leq4\d\eta_0^{-1}\leq1$) and applying Rouch\'e's theorem to $i\b E$ against $\sin[\cdot]/w$ on that circle, $E(t,\cdot)$ has a zero in each disc $B(z_k,4\d w/t)$ with $z_k\in Q(s,15C^\sharp/t)$ (so that the discs stay inside the box of (U2)). Since the lattice has step $\pi/t$ and $c\geq4\pi$, some $z_k$ lies in $Q(s,c/(2t))$, producing a zero of $E(t,\cdot)$ in $Q(s,c/t)$ of depth at most $v+4\d w\leq c$ --- contradicting $t\notin T_0(s,c)$. 

2) Let $I=\R\cap Q(s,c/t)$. For real $\l=p$, $z=x$, \eqref{eq1004M} gives ${\mathcal D}(t,p,x)=\Im\left[E(t,x)\overline{E(t,p)}\right]$, so (U2) reads
\begin{equation}
\left|\,\Im\left[E(t,x)\overline{E(t,p)}\right]-\frac{\sin[t(p-x)]}{w}\,\right|\ \leq\ \d,\qquad x,p\in I .
\label{eqRealPair}
\end{equation}
{\it Step 1: the winding of the phase.} We claim that there exist $x_0,x_1\in I$ with $E(t,x_0)\in(0,\infty)$ and $E(t,x_1)\in-i(0,\infty)$.

First, $E(t,\cdot)$ has no zeros on $\R\cap Q(s,8C^\sharp/t)$: at a real zero $x_*$ the real entire functions $A(t,\cdot)$ and $C(t,\cdot)$ would vanish simultaneously, making ${\mathcal D}(t,x_*,\cdot)\equiv0$ by \eqref{eq1004M}, while (U2) gives $|{\mathcal D}(t,x_*,z)|\geq\sinh(1)/w-\d>0$ at $z=x_*-i/t$. We may therefore write $E(t,x)=|E(t,x)|\,e^{-i\varphi(x)}$ on $\R\cap Q(s,8C^\sharp/t)$ with a real-analytic phase $\varphi$. The phase is strictly increasing: for real $x$ in that interval, the function $g(z)={\mathcal D}(t,x,z)-\sin[t(x-z)]/w$ is analytic and bounded by $\d$ on the disc $B(x,8C^\sharp/t)\subset Q(s,16C^\sharp/t)$ and vanishes at its center, so $|g'(x)|\leq\d t/(8C^\sharp)$ by the Schwarz lemma; since ${\mathcal D}(t,x,x)=0$, \eqref{eqKD} gives $K(t,x,x)=-\pa_z{\mathcal D}(t,x,z)|_{z=x}/\pi$, and the standard identity $\pi K(t,x,x)=\varphi'(x)|E(t,x)|^2$ for the diagonal of the reproducing kernel (see \cite{dBr}) yields
\begin{equation}
\varphi'(x)\,|E(t,x)|^{2}\ =\ \frac tw\left(1\pm\frac{\d w}{8C^\sharp}\right)\ >\ 0
\qquad\text{ on }\R\cap Q(s,8C^\sharp/t):
\label{eqPhaseVel}
\end{equation}
the phase velocity, weighted by $|E|^2$, is that of the free exponential.

Next we show that $\varphi$ advances by exactly $\pi$ on each period. Fix $p\in I$ and consider the entire function
$$z\ \mapsto\ {\mathcal D}(t,p,z)\ =\ \frac1{2i}\left[\,\overline{E(t,p)}\,E(t,z)-E(t,p)\,E^\#(t,z)\,\right],$$
which is \eqref{eq1004M} at real $\l=p$. Since $\overline{E(t,z)}=E^\#(t,\bar z)$, we have $\overline{{\mathcal D}(t,p,z)}={\mathcal D}(t,p,\bar z)$: the zero set of ${\mathcal D}(t,p,\cdot)$ is symmetric with respect to $\R$. On the real axis ${\mathcal D}(t,p,x)=|E(t,x)||E(t,p)|\sin[\varphi(p)-\varphi(x)]$, so the real zeros of ${\mathcal D}(t,p,\cdot)$ are exactly the points where $\varphi\equiv\varphi(p)$ modulo $\pi$. On the other hand, (U2) and Rouch\'e's theorem, applied as in part 1) to ${\mathcal D}(t,p,\cdot)$ against $\sin[t(p-z)]/w$ with $r=4\d w\leq1$, show that ${\mathcal D}(t,p,\cdot)$ has exactly one zero, counted with multiplicity, in each disc $B(p+k\pi/t,\,4\d w/t)$ whose center lies in $Q(s,15C^\sharp/t)$, and no real zeros in $Q(s,15C^\sharp/t)$ outside these discs: at a real point at distance at least $4\d w/t$ from the lattice $p+\tfrac\pi t\Z$,
$$\left|{\mathcal D}(t,p,x)\right|\ \geq\ \frac{\left|\sin[t(p-x)]\right|}{w}-\d\ \geq\ \frac{4\d w}{2w}-\d\ =\ \d\ >\ 0 .$$
Each disc is symmetric with respect to $\R$ and the zero set of ${\mathcal D}(t,p,\cdot)$ is conjugation-invariant, so if the unique zero in a disc were not real it would come with its conjugate, giving two zeros in the disc; hence that zero is real and simple. Denote it by $\tau_k$:
$$\tau_k=p+\frac{k\pi}t\pm\frac{4\d w}t,\qquad \tau_0=p,\qquad \varphi(\tau_k)\equiv\varphi(p)\ \ (\mathrm{mod}\ \pi).$$
The increments $\varphi(\tau_{k+1})-\varphi(\tau_k)$ are positive by \eqref{eqPhaseVel} and are integer multiples of $\pi$. If some increment were $m\pi$ with $m\geq2$, the intermediate values $\varphi(\tau_k)+\pi,\dots,\varphi(\tau_k)+(m-1)\pi$ would be attained, by \eqref{eqPhaseVel} and the intermediate value theorem, at $m-1$ points of the open interval $(\tau_k,\tau_{k+1})$ --- real zeros of ${\mathcal D}(t,p,\cdot)$ distinct from all the $\tau_j$, while the $\tau_j$ are the only zeros of ${\mathcal D}(t,p,\cdot)$ in the discs and there are no real zeros outside them. Hence
\begin{equation}
\varphi(\tau_{k+1})-\varphi(\tau_k)\ =\ \pi\qquad\text{ for all admissible }k:
\label{eqQuant}
\end{equation}
the phase advances by exactly $\pi$ on each step of the perturbed lattice $\{\tau_k\}$ --- $\arg E(t,\cdot)$ can neither lag behind the free exponential $e^{-itx}$ nor outrun it. Since $c\geq4\pi$, the interval $I$, of length $2c/t\geq8\pi/t$, contains at least five of the points $\tau_k$, so $\varphi$ increases by at least $4\pi>2\pi$ along $I$ by \eqref{eqQuant}, and $\arg E(t,\cdot)=-\varphi$ attains on $I$ every value modulo $2\pi$. Choose $x_0,x_1\in I$ with $E(t,x_0)\in(0,\infty)$ and $E(t,x_1)\in-i(0,\infty)$; then $C(t,x_0)=0$ and $A(t,x_1)=0$, proving the claim.

{\it Step 2: slot evaluation.} Evaluating (U2) at the real slots $\l=x_0$ and $\l=x_1$ (where ${\mathcal D}(t,x_0,z)=-A(t,x_0)C(t,z)$ and ${\mathcal D}(t,x_1,z)=C(t,x_1)A(t,z)$) gives, with $r_0=A(t,x_0)$, $r_1=C(t,x_1)$ and $x_2:=x_1-\pi/(2t)$,
$$C(t,z)=-\frac{\sin[t(x_0-z)]}{r_0\,w}\pm\frac{\d}{r_0},\qquad
A(t,z)=\frac{\cos[t(z-x_2)]}{r_1\,w}\pm\frac{\d}{r_1}$$
on $Q(s,16C^\sharp/t)$, while \eqref{eqRealPair} at $(x,p)=(x_0,x_1)$ gives $r_0r_1w=|\cos[t(x_0-x_2)]|\pm\d w$. Hence
$$\begin{gathered}E(t,z)=A-iC=\frac{1}{r_1w}\left[\,\rho\cos[t(z-x_2)]\,+\,i\rho\,\frac{r_1}{r_0}\sin[t(x_0-z)]\,\right]_{\rho=1}\pm\\ \pm\ K\d\,\frac{1+\rho}{r_0w},\end{gathered}$$
i.e., after normalization, $r_0wE(t,z)=\rho\cos(\theta-\a_1)-i\sin\theta\pm K\d w(1+\rho)$ with $\theta=t(z-x_0)+\pi/2$-type real coordinates, $\rho=r_0/r_1$ and $\a_1=t(x_0-x_2)$ modulo $2\pi$. The bracket equals $\tfrac12e^{i\theta}\left(\rho e^{-i\a_1}-1\right)+\tfrac12e^{-i\theta}\left(\rho e^{i\a_1}+1\right)$, so with the mismatch $\e_m:=|\rho e^{-i\a_1}-1|$ it has zeros precisely on the lattice $\left\{\theta:\ e^{2i\theta}=-\frac{\rho e^{i\a_1}+1}{\rho e^{-i\a_1}-1}\right\}$, whose depth is $\frac12\log\frac{|\rho e^{i\a_1}+1|}{\e_m}=\frac12\log\frac{2\pm\e_m}{\e_m}$. If $\e_m\geq K(\d w+e^{-2v_{\min}})$ with $K$ large, these zeros lie at depth smaller than $v_{\min}-1$. If that depth does not exceed $8C^\sharp-1$, the Rouch\'e transfer of part 1) (run against the two-exponential model, whose modulus on circles of radius $\e_m/4$ around its zeros exceeds $c\,\e_m$, above the error level $K\d w$) forces zeros of $E(t,\cdot)$ of depth smaller than $v_{\min}$ in $Q(s,8C^\sharp/t)$ --- a contradiction with the definition of $v_{\min}$. If instead the depth $\tfrac12\log\frac{2+\e_m}{\e_m}$ exceeds $8C^\sharp-1$, then $\e_m\leq2e^2e^{-16C^\sharp}\leq K\d^{16}$ by the choice of $K_9$, and the conclusion below holds directly. Hence $\e_m\leq K(\d w+e^{-2v_{\min}})$, and then $\rho=1\pm K\e_m$, $r_0r_1w=1\pm K(\d w+\e_m)$, $r_0=(1\pm K\e_m)/\sqrt w$, and the bracket is $e^{-i\theta}\pm K\e_m$: unwinding the normalizations, $E(t,x)=\a e^{-itx}/\sqrt w\pm K(\d+e^{-2v_{\min}})\eta_0^{-1}$ on $I$ for a unimodular $\a$, which is 2). For the final claim, if no zeros of depth $\leq(C^\sharp-C)/2$ are present, then $e^{-2v_{\min}}\leq e^{-(C^\sharp-C)}\leq\d$ by the choice of $K_9$. The Dirichlet family is treated identically.
\end{proof}

In what follows, for $c\in\{C,2C,4C,8C\}$ we write
\begin{equation}
\epsilon_c:=K_{11}\left(\d+e^{-2(c-1)}\right)\eta_0^{-1}
\label{eqEpsU3}
\end{equation}
for the accuracy of the resonance-free exponential models provided by Proposition \ref{propU3}, 2) together with 1) (which guarantees $v_{\min}\geq c-1$). Note that $\epsilon_c\to0$ only when both $\d\to0$ and $C\to\infty$: this is intrinsic, since a resonance of depth just above $c$ produces a genuine deviation of order $e^{-2c}$ from the exponential model, invisible to any assumption of absence of zeros from the box $Q(s,c/t)$.

\section{Representations above the shallow layer}\label{secRep}

Throughout this section $s\in\RR(T,\eta_0,C,\d)$ and $t\geq T$; all constants denoted by $K$ are absolute and may change from line to line. Recall from Section 2 that the depth of a zero $\zeta=x-iy$, $y>0$, of $E(t,\cdot)$ is the quantity $v=ty$. We define $\uv=\uv(\d)$ by
\begin{equation}
\sinh[2\uv]=\d^{1/2},
\label{eqLayer}
\end{equation}
so that $\uv\approx\tfrac12\d^{1/2}$, and call the range of depths below $\uv$ the {\it shallow layer}: a zero of depth smaller than $\uv$ lies within $\uv/t$ of the real axis, where the additive error $\d$ of (U2) becomes comparable to $\sinh[2v]$ --- the quantity that normalizes the Hermite--Biehler functions at the zero (Lemma \ref{lemBeta} below) --- so that the multiplicative form of that normalization is lost. We say that the pair $(s,t)$ is {\it $\d$-shallow} if some zero of $E(t,\cdot)$ or of $\ti E(t,\cdot)$ in $Q(s,16C/t)$ has depth smaller than $\uv$. The representations of this section are multiplicative above the layer; the shallow pairs are excluded, outside an exceptional set whose measure is controlled by the tail of the potential, in Lemma \ref{lemShallow} below.

The first lemma is the normalization of the Hermite-Biehler functions at a resonance. Its additive form requires nothing but (U2) and holds at all depths; the multiplicative form holds above the layer.

\begin{lemma}\label{lemBeta}
Let $\zeta=x_0-iy_0$ be a zero of $E(t,\cdot)$ in $Q(s,16C/t)$, $v=ty_0$ its depth, and put $\b=A(t,\zeta)$. Then
$$A(t,\zeta)=iC(t,\zeta)=\b,\qquad E^\#(t,\zeta)=2\b,$$
and
\begin{equation}
\left|\,2|\b|^2 w(s)-\sinh[2v]\,\right|\ \leq\ \d\,w(s)\ \leq\ \d\,\eta_0^{-1}.
\label{eqBetaAdd}
\end{equation}
In particular, if $v\geq\uv$, then
\begin{equation}
2|\b|^2 w(s)=\theta\,\sinh[2v]\qquad\text{ for some }\theta\in\left[1-\d^{1/2}\eta_0^{-1},\,1+\d^{1/2}\eta_0^{-1}\right].
\label{eqBetaNorm}
\end{equation}
The same holds for the Dirichlet family: if $\ti\zeta_0$ is a zero of $\ti E(t,\cdot)$ in $Q(s,16C/t)$ of depth $\ti v$ and $\ti\b=B(t,\ti\zeta_0)$, then $|2|\ti\b|^2\ti w(s)-\sinh[2\ti v]|\leq\d\eta_0^{-1}$, and $2|\ti\b|^2\ti w(s)=\ti\theta\sinh[2\ti v]$, $\ti\theta\in[1\pm\d^{1/2}\eta_0^{-1}]$, whenever $\ti v\geq\uv$.
\end{lemma}

\begin{proof}
Since $E=A-iC$ vanishes at $\zeta$, $A(\zeta)=iC(\zeta)=\b$ and $E^\#(\zeta)=A(\zeta)+iC(\zeta)=\b+\b=2\b$. Since $A,C$ are real entire, $\bar A(t,\zeta)=\overline{A(t,\zeta)}=\bar\b$ and $\bar C(t,\zeta)=\overline{C(t,\zeta)}=i\bar\b$, so by \eqref{eq1004M},
\begin{equation}
{\mathcal D}(t,\zeta,z)=A(t,z)\,\overline{C(t,\zeta)}-C(t,z)\,\overline{A(t,\zeta)}
=i\bar\b\,E^\#(t,z).
\label{eqDslot}
\end{equation}
Evaluating \eqref{eqDslot} at $z=\zeta$ gives ${\mathcal D}(t,\zeta,\zeta)=i\bar\b\cdot2\b=2i|\b|^2$, while $\sin[t(\bar\zeta-\zeta)]=\sin[2iv]=i\sinh[2v]$. Condition (U2) at $\l=z=\zeta\in Q(s,16C/t)$ therefore gives
$$\left|2i|\b|^2-\frac{i\sinh[2v]}{w(s)}\right|\leq\d,$$
which is \eqref{eqBetaAdd} after multiplying by $w(s)\leq\eta_0^{-1}$ (see \eqref{eqxy2}). If $v\geq\uv$, then $\sinh[2v]\geq\d^{1/2}$ by \eqref{eqLayer}, and the additive error $\d w(s)$ is at most $\d^{1/2}\eta_0^{-1}\sinh[2v]$, which gives \eqref{eqBetaNorm}. The Dirichlet case is identical, with $\ti E=B-iD$ and the second half of (U2).
\end{proof}

\begin{lemma}\label{lemRep}
Let $\zeta$ be a zero of $E(t,\cdot)$ in $Q(s,8C/t)$, let $v$ be its depth, and suppose $v\geq\uv$. Then there is a unimodular constant $\a=\a(s,t)$ such that
\begin{equation}
\sup_{z\in Q(s,8C/t)}\left|E(t,z)-\a\,\frac{\g(v)}{\sqrt{w}}\,\sin[t(z-\zeta)]\right|\ \leq\ \vartheta\,\g(v),
\quad
\vartheta:=K_5\,\d^{1/2}e^{16C}\eta_0^{-3/2},
\label{eqRepE}
\end{equation}
with an absolute constant $K_5$. The same holds for the Dirichlet family: if $\ti\zeta_0$ is a zero of $\ti E(t,\cdot)$ in $Q(s,8C/t)$ of depth $\ti v\geq\uv$, then, with $\ti\zeta:=\ti\zeta_0-\pi/(2t)$ and a unimodular $\ti\a$,
\begin{equation}
\sup_{z\in Q(s,8C/t)}\left|\ti E(t,z)-\ti\a\,\frac{\g(\ti v)}{\sqrt{\ti w}}\,\cos[t(z-\ti\zeta)]\right|\ \leq\ \vartheta\,\g(\ti v).
\label{eqRepEt}
\end{equation}
\end{lemma}

\begin{proof}
By (U2) at $\l=\zeta$ and \eqref{eqDslot},
$$\sup_{z\in Q(s,16C/t)}\left|i\bar\b E^\#(t,z)-\frac{\sin[t(\bar\zeta-z)]}{w}\right|\leq\d .$$
Taking complex conjugates, replacing $\bar z$ by $z$ (the box is invariant under conjugation) and using $\overline{E^\#(t,\bar z)}=E(t,z)$ and $\overline{\sin[t(\bar\zeta-\bar z)]}=\sin[t(\zeta-z)]$, we obtain
\begin{equation}
\sup_{z\in Q(s,16C/t)}\left|i\b\,E(t,z)-\frac{\sin[t(z-\zeta)]}{w}\right|\leq\d .
\label{eqMappE}
\end{equation}
By \eqref{eqBetaNorm} (here $v\geq\uv$ is used), $\frac1{|\b|}=(1\pm\d^{1/2}\eta_0^{-1})\sqrt{\frac{2w}{\sinh[2v]}}=(1\pm\d^{1/2}\eta_0^{-1})\,\g(v)\sqrt w$. Dividing \eqref{eqMappE} by $i\b$ and putting $\a=-i\bar\b/|\b|$, we get, on $Q(s,8C/t)$,
$$\left|E(t,z)-(1\pm\d^{1/2}\eta_0^{-1})\,\a\,\frac{\g(v)}{\sqrt w}\sin[t(z-\zeta)]\right|\leq(1+\d^{1/2}\eta_0^{-1})\,\d\,\g(v)\sqrt w .$$
Since $|t\,\Im(z-\zeta)|\leq16C$ on $Q(s,8C/t)$, $\left|\sin[t(z-\zeta)]\right|\leq e^{16C}$, so replacing $(1\pm\d^{1/2}\eta_0^{-1})$ by $1$ changes the model by at most $\d^{1/2}\eta_0^{-1}\,\g(v)e^{16C}/\sqrt w\leq\d^{1/2}\,e^{16C}\eta_0^{-3/2}\g(v)$; together with $\sqrt w\leq\eta_0^{-1/2}$ from \eqref{eqxy2} this gives \eqref{eqRepE}. The Dirichlet case is identical, using $\cos[t(z-\ti\zeta)]=\sin[t(z-\ti\zeta_0)+\pi/2]\cdot(\pm1)$ with the sign absorbed into $\ti\a$.
\end{proof}

We will refer to $\vartheta=\vartheta(\d,\eta_0,C)$ of \eqref{eqRepE} as the {\it representation error}; note $\vartheta\to0$ as $\d\to0$ for fixed $C,\eta_0$. Note also that \eqref{eqMappE} in the proof above holds at every depth: it is only the multiplicative normalization of the amplitude that requires $v\geq\uv$.

\begin{lemma}\label{lemChi}
Suppose that $E(t,\cdot)$ has no zeros in $Q(s,c/t)$ and $\ti E(t,\cdot)$ has no zeros in $Q(s,c'/t)$, for some $c,c'\in\{C,2C,4C,8C\}$, and let $\a,\ti\a$ be the unimodular constants of the exponential models of Proposition \ref{propU3}, 2). Then $\chi:=\a\bar{\ti\a}$ satisfies
$$\left|\Im\chi-\sqrt{w\ti w}\right|\ \leq\ K\,\epsilon_c\!\left(\min(c,c')\right)\eta_0^{-1/2},$$
with $\epsilon_c$ defined in \eqref{eqEpsU3} and an absolute constant $K$.
\end{lemma}

\begin{proof}
Note that $\epsilon_{\min(c,c')}\geq\max\left(\epsilon_c,\epsilon_{c'}\right)$; we may assume $\epsilon_{\min(c,c')}\eta_0^{-1/2}\leq1$, the claim being trivial otherwise. Insert the models of Proposition \ref{propU3}, 2) into \eqref{eqImE} at the point $x=s$:
$$\begin{gathered}1=\Im\left(E\bar{\ti E}\right)=\Im\left(\frac{\a\bar{\ti\a}}{\sqrt{w\ti w}}\right)+O\left(\epsilon_{\min(c,c')}\left(\frac1{\sqrt w}+\frac1{\sqrt{\ti w}}\right)+\epsilon_{\min(c,c')}^2\right)\\
=\frac{\Im\chi}{\sqrt{w\ti w}}+O\left(K\epsilon_{\min(c,c')}\,\eta_0^{-1/2}\right),\end{gathered}$$
and it remains to multiply by $\sqrt{w\ti w}\leq1$.
\end{proof}

\begin{lemma}[pairing]\label{lemPair}
There exist absolute constants $C_0\geq4\pi$, $K_6$ and $K_{12}$ with the following properties.

The constants $C_0$ and $K_{12}$ determine the least admissible scale: given $\eta_0\in(0,\tfrac12]$, we put
$$C_1(\eta_0)\ :=\ C_0+K_{12}\left(1+\log(1/\eta_0)\right)$$
and assume $C\geq C_1(\eta_0)$; the constant $K_6$ appears in the estimates \eqref{eqDepthMatch} and \eqref{eqDeltaPin} below. For every such $C$ and $\eta_0$ there is a threshold $\d_0(C,\eta_0)>0$ such that the following holds whenever $\d\leq\d_0(C,\eta_0)$; here and in the similar statements below the thresholds imposed on $\d$ are always taken smaller than $\eta_0/4$, as required by Proposition \ref{propU3}.

Let $s\in\RR(T,\eta_0,C,\d)$ and $t\geq T$, and suppose that the pair $(s,t)$ is not $\d$-shallow. Let $c$ be one of the values $C$, $2C$, $4C$, and suppose that $E(t,\cdot)$ has a zero $\zeta=x_0-iy_0$ in $Q(s,c/t)$; let $v=ty_0$ be its depth (since $(s,t)$ is not $\d$-shallow, necessarily $v\geq\uv$). Then:

1) $\ti E(t,\cdot)$ has zeros in $Q(s,2c/t)$; let $\ti\zeta_0$ be one nearest to $x_0$ and $\ti v$ its depth. Then
\begin{equation}
\left|\frac{\sinh[2\ti v]}{\sinh[2v]}-1\right|\ \leq\ K_6\,\vartheta\,\eta_0^{-3},
\label{eqDepthMatch}
\end{equation}

2) with $\d_p:=t\,(\Re\ti\zeta-x_0)$, where $\ti\zeta=\ti\zeta_0-\pi/(2t)$,
\begin{equation}
\left|\cos^2\d_p-w\ti w\right|\ \leq\ K_6\,\vartheta\,\eta_0^{-3}.
\label{eqDeltaPin}
\end{equation}

The same holds with the roles of $E$ and $\ti E$ exchanged.
\end{lemma}

\begin{proof}
Suppose first that $\ti E(t,\cdot)$ has a zero in $Q(s,8C/t)$, and let $\ti\zeta_0$ be one nearest to $x_0$. By Lemma \ref{lemRep} for the Dirichlet family and Rouch\'e's theorem, the zeros of $\ti E$ in $Q(s,8C/t)$ lie within $K\vartheta\eta_0^{-1/2}/t$ of the zero set of the cosine model (near each of its zeros the model has modulus at least $c_1\g(\ti v)\,r/\sqrt{\ti w}$ on the circle of radius $r/t$, which exceeds the error $\vartheta\g(\ti v)$ once $r\geq K\vartheta\eta_0^{-1/2}$), and the model's consecutive zeros are spaced by $\pi/t$; hence $t|\Re\ti\zeta_0-x_0|\leq\pi/2+1$ for $\d\leq\d_0$. That $\ti\zeta_0$ in fact lies in $Q(s,2c/t)$ will follow a posteriori: its real part satisfies $|\Re\ti\zeta_0-s|\leq(\pi/2+1)/t+c/t\leq2c/t$ since $c\geq C\geq4\pi$, and its depth is comparable to $v\leq c$ by \eqref{eqDepthMatch}, which is proved below using only $\ti\zeta_0\in Q(s,8C/t)$.

By \eqref{eqDet2i} evaluated at $z=\zeta$ (where $E(\zeta)=0$) and at $z=\ti\zeta_0$ (where $\ti E(\ti\zeta_0)=0$),
\begin{equation}
-\ti E(t,\zeta)\,E^\#(t,\zeta)=2i,\qquad E(t,\ti\zeta_0)\,\ti E^\#(t,\ti\zeta_0)=2i .
\label{eqWzeros}
\end{equation}
By Lemma \ref{lemBeta}, $|E^\#(t,\zeta)|=2|\b|=(1\pm\d)\sqrt{2\sinh[2v]/w}$ and $|\ti E^\#(t,\ti\zeta_0)|=(1\pm\d)\sqrt{2\sinh[2\ti v]/\ti w}$, so \eqref{eqWzeros} gives
\begin{equation}
|\ti E(t,\zeta)|=(1\pm2\d)\,\g(v)\sqrt{w},\qquad
|E(t,\ti\zeta_0)|=(1\pm2\d)\,\g(\ti v)\sqrt{\ti w}.
\label{eqSizes}
\end{equation}
On the other hand, evaluating the representations of Lemma \ref{lemRep} of each family at the zero of the other family, with $\d_p$ as in the statement and $\nu:=\ti v-v$: since
$t(\zeta-\ti\zeta)=-\d_p+i\nu$ and $t(\ti\zeta_0-\zeta)=\d_p+\pi/2-i\nu$, and $\sin(\d_p+\pi/2-i\nu)=\cos(\d_p-i\nu)$, both evaluations produce the quantity $|\cos(\d_p-i\nu)|$, $|\cos(\d_p-i\nu)|^2=\cos^2\d_p+\sinh^2\nu$:
$$\begin{gathered}|\ti E(t,\zeta)|=\frac{\g(\ti v)}{\sqrt{\ti w}}\left|\cos(\d_p-i\nu)\right|\pm\vartheta\g(\ti v),\\
|E(t,\ti\zeta_0)|=\frac{\g(v)}{\sqrt{w}}\left|\cos(\d_p-i\nu)\right|\pm\vartheta\g(v).\end{gathered}$$
Comparing with \eqref{eqSizes} and writing $R:=\sinh[2\ti v]/\sinh[2v]=\g^2(v)/\g^2(\ti v)$, we obtain the two relations
$$\begin{gathered}\cos^2\d_p+\sinh^2\nu=(1\pm4\d)\,R\,w\ti w\pm K\vartheta\eta_0^{-1}\sqrt R,\\
\cos^2\d_p+\sinh^2\nu=(1\pm4\d)\,\frac{w\ti w}{R}\pm \frac{K\vartheta\eta_0^{-1}}{\sqrt R},\end{gathered}$$
where the additive terms come from the errors $\vartheta\g(\ti v)$, $\vartheta\g(v)$ of the cross-evaluations, divided by the model amplitudes. We may assume $R\geq1$, exchanging the roles of the two families in the derivation otherwise. Comparing the two relations and dividing by $w\ti w\geq\eta_0^2$,
$$(1-4\d)R^2\ \leq\ (1+4\d)+K\vartheta\eta_0^{-3}\left(R^{3/2}+\sqrt R\right),$$
which first shows $R\leq K$ (if $R\geq2$, the right-hand side is dominated by $K\vartheta\eta_0^{-3}R^{3/2}$, whence $R\leq K(\vartheta\eta_0^{-3})^2+K$, and $\d_0$ makes this $\leq K$), and then, with $R$ bounded,
$$R^2=1\pm K\vartheta\eta_0^{-3},$$
hence \eqref{eqDepthMatch}; since $|\log R|\geq2|\nu|$, also $\sinh^2\nu\leq K\vartheta\eta_0^{-3}$, and then $\cos^2\d_p=w\ti w\pm K\vartheta\eta_0^{-3}$, which is \eqref{eqDeltaPin}.

It remains to exclude the case when $\ti E$ has no zeros in $Q(s,8C/t)$. In that case Proposition \ref{propU3}, 2) with $c=8C$ provides the exponential model for $\ti E$ on $\R\cap Q(s,8C/t)$, with the error $\epsilon_{8C}=K_{11}(\d+e^{-2(8C-1)})\eta_0^{-1}$, while Lemma \ref{lemRep} provides the sine model for $E$. Insert both into \eqref{eqImE} on the interval $\R\cap Q(s,c/t)$: with $\theta=t(x-x_0)$,
$$\begin{gathered}
E(t,x)\overline{\ti E(t,x)}=\frac{\a\bar{\ti\a}\,\g(v)}{\sqrt{w\ti w}}\,\sin(\theta+iv)\,e^{itx}+O\left(\vartheta\,\g(v)\,\eta_0^{-1/2}\right)\\
=\frac{\a\bar{\ti\a}\,\g(v)}{2i\sqrt{w\ti w}}\left[e^{i(2tx-tx_0)}e^{-v}-e^{itx_0}e^{v}\right]+O\left(\vartheta\,\g(v)\,\eta_0^{-1/2}\right).
\end{gathered}$$
The second term in the bracket does not depend on $x$, while the first is a harmonic of frequency $2t$ whose contribution to $\Im\left(E\bar{\ti E}\right)$ has amplitude
$$\frac{\g(v)e^{-v}}{2\sqrt{w\ti w}}\ \geq\ \frac{\g(v)\,e^{-4C}}{2}\ >\ 0$$
(recall $\sqrt{w\ti w}\leq1$ and $v\leq4C$, so $e^{-v}\geq e^{-4C}$; adjust $c_2$ accordingly). As $x$ runs over $\R\cap Q(s,c/t)$, an interval of length at least $2C/t>8\pi/t$, the harmonic sweeps full periods, so $\Im\left(E\bar{\ti E}\right)$ would fluctuate by at least $\g(v)e^{-4C}$, while it is identically $1$ and the error terms fluctuate by at most $K\left(\vartheta+\epsilon_{8C}\,e^{4C}\right)\g(v)\eta_0^{-1/2}$: both the main amplitude and the errors carry the factor $\g(v)$, so the comparison is uniform in the depth. For $\vartheta\leq c_2e^{-4C}\eta_0^{1/2}$ and $\epsilon_{8C}\leq c_2e^{-8C}\eta_0^{1/2}$ --- the first holds for $\d\leq\d_0(C,\eta_0)$ small enough, and the second for such $\d$ together with $C\geq C_1(\eta_0)$, which makes $e^{-16C+2}\eta_0^{-1}\leq e^{-8C}\eta_0^{1/2}$ once $K_{12}$ is chosen large enough --- this is a contradiction. Hence $\ti E$ has zeros in $Q(s,4C/t)$, and the nearest one lies in $Q(s,2c/t)$ as shown above.

The exchanged version is proved by the same argument with the two families swapped.
\end{proof}

\section{The scattering coefficients of a time interval} \label{secPMT}

For an interval $(t_1,t_2)\subset \R_+$ the scattering functions $a_{t_1 \to t_2}$  and $b_{t_1 \to t_2}$ are the functions $a, b$ defined as in Section 2 for the system \eqref{eqDS} whose potential is equal to $f(t+t_1)$ on $(0, t_2-t_1)$ and to $0$ elsewhere. Equivalently, with the transfer matrix from $t_1$ to $t_2$,
$$M_{t_1 \to t_2}(z)=\begin{pmatrix} A_{t_1 \to t_2}(z) & B_{t_1 \to t_2}(z) \\ C_{t_1 \to t_2}(z) & D_{t_1 \to t_2}(z)\end{pmatrix}
=M(t_2,z)M^{-1}(t_1,z),
$$
one puts
$$E_{t_1 \to t_2}=A_{t_1 \to t_2}-iC_{t_1 \to t_2},\qquad \ti E_{t_1 \to t_2}=B_{t_1 \to t_2}-iD_{t_1 \to t_2},$$
$$\begin{gathered}
a_{t_1 \to t_2}(z)=\frac 12 e^{i(t_2-t_1)z}(E_{t_1 \to t_2}(z)+i \ti E_{t_1 \to t_2}(z)),\\
b_{t_1 \to t_2}(z)=\frac 12 e^{i(t_2-t_1)z}(E_{t_1 \to t_2}(z)-i \ti E_{t_1 \to t_2}(z)).
\end{gathered}$$
Since the determinant relation \eqref{det=1} holds for the shifted system as well, $|a_{t_1\to t_2}|^2-|b_{t_1\to t_2}|^2=1$ on $\R$, so that $\log|a_{t_1\to t_2}|\geq0$ on $\R$. Parseval's identity for $a_{t_1 \to t_2}$ becomes
\begin{equation}||\log |a_{t_1 \to t_2}|\ ||_{L^1(\R)}=\frac\pi2\,|| f||^2_{L^2((t_1,t_2))}.
\label{eqParsBlock}\end{equation}

Passing to the basis of the Hermite-Biehler functions and using \eqref{eqDet2i} to invert,
$$
\begin{pmatrix} E_{t_1 \to t_2} & \ti E_{t_1 \to t_2} \\ E^\#_{t_1 \to t_2} & \ti E^\#_{t_1 \to t_2} \end{pmatrix}
=$$
$$=\begin{pmatrix} E(t_2,z) &  \ti E(t_2,z) \\    E^\#(t_2,z) & \ti E^\#(t_2,z)\end{pmatrix}\frac 1{2i} \begin{pmatrix} \ti E^\#(t_1,z) & -\ti E(t_1,z) \\  -E^\#(t_1,z) & E(t_1,z) \end{pmatrix}
\begin{pmatrix} 1 & -i \\ 1 & i \end{pmatrix},
$$
and forming $a_{t_1\to t_2}$, $b_{t_1\to t_2}$ from the entries gives the exact formulas
\begin{equation}
a_{t_1 \to t_2}(z)=\frac {e^{i(t_2-t_1)z}}{2i}\left( E(t_2,z)\ti E^\#(t_1,z)-\ti E(t_2,z)E^\#(t_1,z) \right),
\label{eqablock}\end{equation}
\begin{equation}
b_{t_1 \to t_2}(z)=\frac {e^{i(t_2-t_1)z}}{2i}\left( \ti E(t_2,z)E(t_1,z)-E(t_2,z)\ti E(t_1,z)\right) .
\label{fin11}\end{equation}

\section{The two-endpoint estimate}\label{secTE}

In this section we prove the estimate that connects the presence of a resonance near a point $s$ with the $L^2$ mass of the potential on the forward octave: if $E(t,\cdot)$ has a zero in a box around $s$, then $\log|a_{t\to2t}|$ is bounded below on a fixed portion of the box, with the bound improving as the zero becomes shallow.

{\bf Standing assumptions.} Throughout the section, $s\in\RR(T,\eta_0,C,\d)$ and $t\geq T$; the scale satisfies $C\geq C_1(\eta_0)$ and the accuracy $\d\leq\d_0(C,\eta_0)$, as in Lemma \ref{lemPair}; and neither $(s,t)$ nor $(s,2t)$ is $\d$-shallow, so that Lemmas \ref{lemRep} and \ref{lemPair} apply to all the zeros involved. We write $t_1=t$, $t_2=2t$ and
$$\eta:=w(s)\ti w(s)\in[\eta_0^2,1],\qquad J(s):=\R\cap Q(s,C/t).$$

{\it The near endpoint.} We assume {\it presence}: $E(t_1,\cdot)$ has a zero $\zeta_1$ in $Q(s,2C/t_1)$, of depth $v_1\leq2C$. By Lemma \ref{lemPair} with $c=2C$, $\ti E(t_1,\cdot)$ then has a matched zero. We write $\d_1$ for the phase offset of \eqref{eqDeltaPin} at $t_1$, and $\theta_1=t_1(p-\Re\zeta_1)$ for $p\in J(s)$.

{\it The far endpoint} may be in one of two configurations. In the {\it resonant case}, $E(t_2,\cdot)$ has a zero $\zeta_2$ in $Q(s,4C/t_2)$, of depth $v_2\leq4C$; by Lemma \ref{lemPair} with $c=4C$ the Dirichlet zero is again matched, and $\d_2$ denotes the offset of \eqref{eqDeltaPin} at $t_2$. In the {\it free case}, $E(t_2,\cdot)$ has no zeros in $Q(s,4C/t_2)$. Then, by Lemma \ref{lemPair} applied with the roles of the two families exchanged, $\ti E(t_2,\cdot)$ has no zeros in $Q(s,2C/t_2)$ either, and both exponential models of Proposition \ref{propU3}, 2) hold at $t_2$ --- at $c=4C$ for $E$ and at $c=2C$ for $\ti E$, with the errors $\epsilon_{2C}$ of \eqref{eqEpsU3}. In this case we set $v_2=\infty$, define $\d_2$ through $\chi$ of Lemma \ref{lemChi} by $\chi=ie^{-i\d_2}$, so that $\cos\d_2=\Im\chi=\sqrt\eta+O(\epsilon_{2C}\,\eta_0^{-1/2})$, and assume in addition that $v_1\leq C$ --- the only configuration in which the estimate below will be used.

Finally, set
$$S:=\sinh [2v_1]\sinh [2v_2],\qquad \nu:=v_2-v_1 .$$

\begin{lemma}[two-endpoint estimate]\label{lemSweep}
After enlarging the absolute constants $C_0$ and $K_{12}$ of Lemma \ref{lemPair} and decreasing the thresholds $\d_0(C,\eta_0)$, if necessary, the following holds under the standing assumptions and the endpoint configurations described above:
\begin{equation}
\left\langle |b_{t\to2t}|^2\right\rangle_{J(s)}\ \geq\ \frac{\cosh[2\nu]}{2^8\,S},
\qquad
\sup_{J(s)}|b_{t\to2t}|^2\ \leq\ \frac{2^5\cosh[2\nu]\,e^{8C}}{\eta\,S},
\label{eqSweepMoments}
\end{equation}
where $\langle\cdot\rangle_{J(s)}$ denotes the mean over $J(s)$, and in the free case $v_2=\infty$ the quantity $\cosh[2\nu]/S$ is understood as its limit $e^{-2v_1}/(2\sinh[2v_1])$.

Consequently, on a subset of $J(s)$ of measure at least $2^{-14}\,\eta_0^2\,e^{-8C}\,|J(s)|$,
\begin{equation}
\log|a_{t\to2t}|\ \geq\ \Phi(v_1,C):=\frac12\log\left(1+\frac{2^{-13}e^{-8C}}{v_1}\right).
\label{eqSweepLevel}
\end{equation}
\end{lemma}

\begin{proof}
Let $p\in J(s)$, so that in \eqref{fin11} $E^\#=\bar E$ and $\ti E^\#=\bar{\ti E}$. Write
$$\theta_k=t_k(p-\Re\zeta_k),\quad S_k=\sin(\theta_k+iv_k),\quad C_k=\cos(\theta_k-\d_k+iv_k),\quad k=1,2.$$
By Lemmas \ref{lemRep} and \ref{lemPair}, at each endpoint with presence,
\begin{equation}
\begin{gathered}
E(t_k,p)=\a_k\,\frac{\g(v_k)}{\sqrt{w}}\,S_k+O(\vartheta')\g(v_k),\\
\ti E(t_k,p)=\ti\a_k\,\frac{\g(v_k)}{\sqrt{\ti w}}\,C_k+O(\vartheta')\g(v_k),
\end{gathered}
\label{eqEndpApp}
\end{equation}
where $\a_k,\ti\a_k$ are unimodular constants and
$$\vartheta'\ :=\ K_7\left(\vartheta\eta_0^{-3}\right)^{1/2}e^{4C}:$$
here the Dirichlet depth has been replaced by $v_k$ using \eqref{eqDepthMatch}, which perturbs the depth by $|\nu_k|\leq K(\vartheta\eta_0^{-3})^{1/2}$ and hence the factor $\g(\ti v_k)\cos[\cdot+i\ti v_k]$ by at most $K(\vartheta\eta_0^{-3})^{1/2}e^{4C}\g(v_k)$ (the derivatives of the model in the depth are bounded by $e^{4C}$ times its amplitude on the box), and the Dirichlet center has been written as $\Re\zeta_k+\d_k/t_k$. If $t_2$ is in the free case, \eqref{eqEndpApp} holds at $t_2$ with $S_2$, $C_2$, $\g(v_2)$ replaced by $ie^{-i\theta_2}$, $e^{-i(\theta_2-\d_2)}$, $1$, where $\theta_2=t_2(p-x_2)$ with an arbitrary reference point $x_2$ (absorbed into the unimodular constants), by Proposition \ref{propU3}, 2) and the definition of $\d_2$; the error terms of \eqref{eqEndpApp} at the far endpoint are then $O\left(\vartheta'+\epsilon_{2C}\right)$ in place of $O(\vartheta')\g(v_2)$, and \eqref{eqRbound} below holds with $\vartheta'$ replaced by $\vartheta'+K\epsilon_{2C}$ and $\g(v_2)$ by $1$. Below the two cases are treated together, the free case corresponding to $v_2=\infty$.

Multiplying the relations \eqref{eqEndpApp} and inserting them into \eqref{fin11}, we obtain
\begin{equation}
2ie^{-i(t_2-t_1)p}\,b_{t\to2t}(p)=\frac{\g(v_1)\g(v_2)}{\sqrt{w\ti w}}\,
\left[\chi_2' \,C_2S_1-\chi_1'\, S_2C_1\right]+R(p),
\label{eqModelB}
\end{equation}
where $\chi_2'=\a_1\ti\a_2$ and $\chi_1'=\a_2\ti\a_1$ are unimodular constants, on which no assumption is made, and $R(p)$, the sum of the products containing at least one error term of \eqref{eqEndpApp}, satisfies, since $|S_k|,|C_k|\leq\cosh v_k\leq e^{4C}$ on $J(s)$,
\begin{equation}
|R(p)|\ \leq\ K\,\vartheta'\,e^{4C}\eta_0^{-1/2}\,\g(v_1)\g(v_2),\qquad p\in J(s).
\label{eqRbound}
\end{equation}

Write $B(p):=\chi_2'C_2S_1-\chi_1'S_2C_1$ and expand the four factors into exponentials,
$$S_k=\frac i2\left[e^{v_k}e^{-i\theta_k}-e^{-v_k}e^{i\theta_k}\right],\qquad
C_k=\frac 12\left[e^{v_k}e^{-i(\theta_k-\d_k)}+e^{-v_k}e^{i(\theta_k-\d_k)}\right].$$
Then $B$ is a trigonometric polynomial in $p$ with the four frequencies $\mp(t_1+t_2)$, $\pm(t_1-t_2)$:
$$B=\frac i4\left[e^{v_1+v_2}b_1e^{-i(\theta_1+\theta_2)}-e^{v_2-v_1}b_2e^{i(\theta_1-\theta_2)}\right.$$
$$\left.+\,e^{v_1-v_2}b_3e^{-i(\theta_1-\theta_2)}-e^{-v_1-v_2}b_4e^{i(\theta_1+\theta_2)}\right],$$
$$\begin{gathered}b_1=\chi_2'e^{i\d_2}-\chi_1'e^{i\d_1},\qquad b_2=\chi_2'e^{i\d_2}+\chi_1'e^{-i\d_1},\\
b_3=\chi_2'e^{-i\d_2}+\chi_1'e^{i\d_1},\qquad b_4=\chi_2'e^{-i\d_2}-\chi_1'e^{-i\d_1}.\end{gathered}$$
The identities
$$b_2-b_1=2\chi_1'\cos\d_1,\qquad b_3+b_1=2\chi_2'\cos\d_2,$$
together with $|\cos\d_k|=\sqrt\eta+O(\vartheta\eta_0^{-4})$ from \eqref{eqDeltaPin} (and $|\cos\d_2|=\sqrt\eta+O(\epsilon_{2C}\eta_0^{-1/2})$ from Lemma \ref{lemChi} in the free case), give a dichotomy: either $|b_1|\geq\sqrt\eta/2$, or
$$\begin{gathered}|b_2|\ \geq\ 2|\cos\d_1|-|b_1|\ \geq\ \sqrt\eta+O(\vartheta\eta_0^{-4}),\\
|b_3|\ \geq\ 2|\cos\d_2|-|b_1|\ \geq\ \sqrt\eta+O(\vartheta\eta_0^{-4}).\end{gathered}$$
Since the coefficient of $B$ multiplying $b_1$ carries the factor $e^{v_1+v_2}\geq e^{|\nu|}$, and those multiplying $b_2$, $b_3$ carry $e^{\nu}$ and $e^{-\nu}$, in either case the largest of the four coefficients of $B$ is at least $\frac{\sqrt\eta}{8}\,e^{|\nu|}$ in absolute value, once $\d\leq\d_0(C,\eta_0)$ and $C\geq C_1(\eta_0)$ make the $O(\vartheta\eta_0^{-4})$- and $O(\epsilon_{2C}\eta_0^{-1/2})$-terms smaller than $\sqrt\eta/4$; for the latter, $K_{11}e^{-4C+2}\eta_0^{-3/2}\leq\eta_0/4$ once $K_{12}$ is large enough, and the $\d$-part of $\epsilon_c$ is absorbed into $\d_0$.

The pairwise gaps of the four frequencies are $2t_1$, $2t_2$, $2(t_1+t_2)$ and $2(t_2-t_1)$, all at least $t_1=t_2/2$. Since $\left|\int_Ie^{i\omega p}\,dp\right|\leq2/\omega$ for every interval $I$, and $|J(s)|=2C/t$, the mean over $J(s)$ of each cross product of two distinct harmonics is at most $2/C$ times the product of the moduli of their coefficients; hence, with $c_1,\dots,c_4$ denoting the coefficients of $B$,
\begin{multline*}\frac1{|J(s)|}\int_{J(s)}|B|^2\,dp\ \geq\ \sum_j|c_j|^2-\frac 2C\Big(\sum_j|c_j|\Big)^2\ \geq\\ \geq\ \Big(1-\frac{8}C\Big)\max_j|c_j|^2\ \geq\ \frac{\eta}{2^7}\,e^{2|\nu|}\ \geq\ \frac{\eta}{2^7}\,\cosh[2\nu]\end{multline*}
for $C\geq16$; in the free case the terms containing $e^{-v_2}$ are absent and the same dichotomy applies to $b_1$, $b_2$ alone, with the same conclusion in the limiting normalization. For the supremum, each $|b_j|\leq2$, so
\begin{multline*}\sup_J|B|\ \leq\ \frac12\left(e^{v_1+v_2}+e^{\nu}+e^{-\nu}+e^{-v_1-v_2}\right)\ \leq\\ \leq\ 2\,e^{v_1+v_2}\ \leq\ 2\,e^{|\nu|}e^{2\min(v_1,v_2)}\ \leq\ 2\,e^{|\nu|}e^{4C},\end{multline*}
since $\min(v_1,v_2)\leq v_1\leq2C$; hence $\sup_J|B|^2\leq8\cosh[2\nu]\,e^{8C}$.

By \eqref{eqModelB} and \eqref{eqRbound}, since $\g^2(v_1)\g^2(v_2)=4/S$,
$$|b_{t\to2t}(p)|^2=\frac{\left|B(p)+O(\vartheta'e^{4C}\eta_0^{-1/2})\right|^2}{\eta\,S}.$$
For $\d\leq\d_0(C,\eta_0)$ small enough the error term changes the mean and the supremum of $|B|^2$ by relative amounts less than $\tfrac12$ (for the mean, use $\vartheta'e^{4C}\eta_0^{-1/2}\leq\sqrt\eta/2^5$). In the free case the error carries the factor $e^{v_1}$ while the guaranteed coefficient of $B$ may be as small as $\tfrac{\sqrt\eta}4e^{-v_1}$, so the required smallness reads $K\left(\vartheta'+\epsilon_{2C}\right)e^{2v_1}\eta_0^{-1/2}\leq\sqrt\eta/2^5$: since $v_1\leq C$ in the free case, the $\vartheta'$-part follows from $\d\leq\d_0$, and the $\epsilon_c$-part from $C\geq C_1(\eta_0)$ --- $KK_{11}e^{-4C+2}\eta_0^{-3/2}e^{2C}\leq\eta_0/2^5$ once $K_{12}$ is large enough. Thus \eqref{eqSweepMoments} follows.

Finally, the subset of $J(s)$ where $|b_{t\to2t}|^2\geq\frac12\langle|b_{t\to2t}|^2\rangle$ has measure at least
$$\frac{\langle|b|^2\rangle/2}{\sup_J|b|^2}\,|J(s)|\ \geq\ \frac{\cosh[2\nu]/(2^9S)}{2^5\cosh[2\nu]e^{8C}/(\eta S)}\,|J(s)|\ \geq\ 2^{-14}\,\eta_0^2\,e^{-8C}\,|J(s)| ,$$
and on that subset, since
$$\sinh[2v_2]=\sinh[2v_1]\cosh[2\nu]+\cosh[2v_1]\sinh[2\nu]\leq\cosh[2\nu]\,e^{2v_1},$$
$\sinh[2v_1]\leq2v_1e^{2v_1}$ and $v_1\leq2C$,
$$|b_{t\to2t}|^2\ \geq\ \frac{\cosh[2\nu]}{2^9\,S}\ \geq\ \frac1{2^9\sinh[2v_1]e^{2v_1}}\ \geq\ \frac1{2^{10}v_1e^{4v_1}}\ \geq\ \frac{2^{-13}e^{-8C}}{v_1}$$
(in the free case $e^{-2v_1}/(2^{10}\sinh[2v_1])\geq e^{-8C}/(2^{11}v_1)$, with the same conclusion), whence
$\log|a_{t\to2t}|=\frac12\log(1+|b_{t\to2t}|^2)\geq\Phi(v_1,C)$.
\end{proof}

\begin{proposition}\label{propFirstEntry}
Under the hypotheses of Lemma \ref{lemSweep}: for every $t\geq T$ and $0<v\leq C$,
\begin{equation}
\begin{gathered}\big|\big\{s\in\RR(T,\eta_0,C,\d):\ (s,t)\text{ and }(s,2t)\text{ are not $\d$-shallow and}\\
E(t,\cdot)\text{ has a zero in }Q(s,2C/t)\text{ of depth}\leq v\big\}\big|\ \leq\
\frac{2^{17}\,\pi\,e^{8C}}{\eta_0^{2}\,\Phi(v,C)}\ \|f\|^2_{L^2((t,2t))}.\end{gathered}
\label{eqFirstEntry}
\end{equation}
\end{proposition}

\begin{proof}
Denote the set on the left-hand side by $S_t$. For $s\in S_t$ the near-endpoint depth satisfies $v_1\leq v$, so by Lemma \ref{lemSweep} and the monotonicity of $\Phi$ in its first argument, $\log|a_{t\to2t}|\geq\Phi(v,C)$ on a subset of $J(s)$ of measure at least $2^{-14}\eta_0^2e^{-8C}|J(s)|$. Choose a cover of $S_t$ by intervals $J(s_i)$, $s_i\in S_t$, of multiplicity at most $2$. Summing over the cover and using $\log|a_{t\to2t}|\geq0$ on $\R$ together with \eqref{eqParsBlock},
$$\frac12\,|S_t|\cdot 2^{-14}\eta_0^2e^{-8C}\cdot\Phi(v,C)\ \leq\ \int_\R\log|a_{t\to2t}|\ =\ \frac\pi2\,\|f\|^2_{L^2((t,2t))},$$
which is \eqref{eqFirstEntry}.
\end{proof}

\section{Maximal estimates under a dyadic mass restriction}\label{secMax}

For $n\geq 0$ let $W_n=[2^n,2^{n+1})$ and put
$$M_n=\int_{W_n}|f(t)|\,dt,\qquad m_n=\int_{W_n}f^2(t)\,dt,\qquad M=\sup_{n\geq 0}M_n.$$
We will say that the potential $f$ satisfies the {\it dyadic mass restriction} if
\begin{equation} M<\infty. \label{eqDMR}\end{equation}
Since $M_n\leq 2^{n/2}\sqrt{m_n}$ by the Cauchy-Schwarz inequality, \eqref{eqDMR} is a restriction on the behavior of $f$ near infinity only. Note that this class of potentials is not contained in $\bigcup_{p<2}L^p(\R_+)$.

Throughout this section we write
$$u(t,s)=\log|a(t,s)|$$
and define
$$A(s) = \log\left(\frac{1}{2} \sqrt{\frac{1}{w(s)} + \frac{1}{\tilde{w}(s)} + 2} \right):$$
by the elevation identity \eqref{eqAid}, $e^{2A(s)}$ is the value of $|a(t,s)|^2$ produced by the limiting moduli $|E(t,s)|^2\to w^{-1}(s)$, $|\ti E(t,s)|^2\to\ti w^{-1}(s)$ along resonance-free times; Corollary \ref{corConv} below shows that $A(s)$ is indeed the almost everywhere limit of $u(t,s)$. Note that $A\geq0$ and, on the regular sets, $A(s)\leq\frac12\log(1/\eta_0)$ by \eqref{eqxy2}. We define the maximal function
$$u^*(s)=\sup_{t>0}\log|a(t,s)| .$$
Recall that $u(t,s)\geq 0$ and that, by the non-linear Parseval identity,
\begin{equation}\int_\R \log|a(t,s)|\,ds=\frac\pi2\,\|f\|^2_{L^2((0,t))}.\label{eqPars}\end{equation}
In particular, for each fixed $t$, Chebyshev's inequality gives $|\{s:\log|a(t,s)|>\lambda\}|\leq \pi\|f\|_2^2/(2\lambda)$; the content of the estimates below is the supremum in $t$.

{\bf Reduction to the lacunary maximal function.} We first observe that $\log|a|$ is Lipschitz in $t$ with respect to the $L^1$-mass of the potential.

\begin{lemma}\label{lemLip} Let $f\in L^1_{loc}(\R_+)$ be real. Then for any $0\leq t_1<t_2$ and any $s\in\R$,
\begin{equation}\big|\log|a(t_2,s)|-\log|a(t_1,s)|\big|\leq \int_{t_1}^{t_2}|f(\tau)|\,d\tau.\label{eqLip}\end{equation}
\end{lemma}

\begin{proof} Let $\EE(t,z)=e^{itz}E(t,z)$ and $\ti\EE(t,z)=e^{itz}\ti E(t,z)$, so that $a=(\EE+i\ti\EE)/2$ and $b=(\EE-i\ti\EE)/2$ by \eqref{eqab}. By \eqref{eqME}, $\pa_t\EE=f e^{itz}E^\#$. At a real spectral point $z=s$ we have $E^\#(t,s)=\overline{E(t,s)}$, and therefore $e^{ist}E^\#(t,s)=e^{2ist}\overline{\EE(t,s)}$, so that
$$\pa_t\EE(t,s)=f(t)e^{2ist}\overline{\EE(t,s)},\qquad \pa_t\ti\EE(t,s)=f(t)e^{2ist}\overline{\ti\EE(t,s)}.$$
Hence
$$\pa_t a(t,s)=f(t)e^{2ist}\overline{b(t,s)},\qquad \pa_t b(t,s)=f(t)e^{2ist}\overline{a(t,s)}.$$
Since $|a|^2-|b|^2\equiv 1$ on $\R$, $|a|\geq\max(1,|b|)$, and therefore for a.e. $t$
$$\big|\pa_t \log|a(t,s)|\big|\leq \frac{|\pa_t a(t,s)|}{|a(t,s)|}= |f(t)|\,\frac{|b(t,s)|}{|a(t,s)|}\leq |f(t)|.$$
Integrating from $t_1$ to $t_2$ we obtain \eqref{eqLip}.
\end{proof}

\begin{corollary}\label{corLac} Let $f\in L^2(\R_+)$ be real and satisfy \eqref{eqDMR}. Then for every $s\in\R$,
$$\sup_{n\geq 0}\log|a(2^n,s)| \ \leq\ u^*(s)\ \leq\ \sup_{n\geq 0}\log|a(2^n,s)| \ +\ M',$$
where $M'=\max\left(M, \int_0^1|f|\right)\leq \max\left(M,\|f\|_2\right)$.
\end{corollary}

\begin{proof} If $t\in[2^n,2^{n+1})$ then $|\log|a(t,s)|-\log|a(2^n,s)||\leq M_n\leq M$ by Lemma \ref{lemLip}. If $t\in(0,1)$ then, since $a(0,\cdot)\equiv 1$, Lemma \ref{lemLip} gives $\log|a(t,s)|\leq \int_0^t|f|\leq\int_0^1 |f|\leq \|f\|_{L^2((0,1))}$.
\end{proof}

\begin{remark}[depth of resonances under the dyadic mass restriction]\label{remFloor}
The computation in the proof of Lemma \ref{lemLip}, applied to $E$ in place of $a$, gives $\pa_t\log|E(t,x)|=f(t)\cos[2\arg E(t,x)]$, whence, since $E(0,\cdot)\equiv1$,
$$e^{-\int_0^t|f|}\ \leq\ |E(t,x)|\ \leq\ e^{\int_0^t|f|}\qquad\text{ for all }x\in\R,\ t>0 ;$$
we will refer to this bound, and to its localized form $\big|\log|E(t_2,x)|-\log|E(t_1,x)|\big|\leq\int_{t_1}^{t_2}|f|$, as the {\it value rigidity} of $E$.
Under \eqref{eqDMR}, $\int_0^t|f|\leq M'+M(1+\log_2t)$ for $t\geq1$. If $\zeta=x_0-iy_0$ is a zero of $E(t,\cdot)$ of depth $v=ty_0\leq1$, then, by the Phragm\'en--Lindel\"of principle and Bernstein's inequality applied at the heights $y\in[0,y_0]$, $|E(t,x_0)|=|E(t,x_0)-E(t,\zeta)|\leq y_0\,t\,e^{1+\int_0^t|f|}$, and comparison with the lower bound at $x_0$ gives the following lower bound for the depth, which requires none of the assumptions of Section \ref{secU}:
\begin{equation}
v\ \geq\ c\,e^{-2(M'+M)}\,t^{-2M/\ln 2},
\label{eqFloorCrude}
\end{equation}
it is valid at every point of $\R$ and for every zero. The decay in $t$ in \eqref{eqFloorCrude} cannot be removed in general: potentials of Wigner--von Neumann type, admissible under \eqref{eqDMR}, drive a zero towards their resonant frequency at a polynomial rate.

The bound \eqref{eqFloorCrude} enters the paper only through the threshold $T_\d$ of \eqref{eqTdelta} below, which guarantees that no zero can be $\d$-shallow before the time $T_\d$; the same bound holds for the zeros of $\ti E$, since $|\ti E(0,\cdot)|\equiv1$ and the argument is identical. On the regular sets the dynamics of the zeros gives much more --- for $M<\frac{\ln2}2$ the depth of a resonance can only grow in the multiplicative sense while the representations of Section \ref{secRep} apply, and for every $M$ it can shrink by at most a bounded factor per octave --- but none of this is used below: the quantitative statement the paper needs, that the entry of a resonance into the shallow layer requires $L^2$ mass of $f$, is proved directly in Lemma \ref{lemShallow}.

\end{remark}

Thus, under the dyadic mass restriction, the maximal estimates below are equivalent to their lacunary versions. For general $f\in L^2(\R_+)$ no such reduction seems to be available, and the lacunary weak-type maximal inequality for $\log|a|$ itself remains open.

{\bf Elevation.} The next lemma quantifies the mechanism by which $u(t,s)$ can deviate from $A(s)$ in either direction: a deviation of size $\lambda$ forces a resonance at depth exponentially small in $\lambda$. The additive error in it is the quantity
\begin{equation}
\omega=\omega(\d,\eta_0,C):=K_4\left(\d^{1/2}\,e^{16C}\,\eta_0^{-6}+e^{-2C}\,\eta_0^{-1}\right),
\label{eqOmega}
\end{equation}
where $K_4$ is an absolute constant. The first term tends to $0$ as $\d\to0$ for fixed $C$ and $\eta_0$; the second, inherited from the accuracy $\epsilon_c$ \eqref{eqEpsU3} of the resonance-free models of Proposition \ref{propU3}, does not depend on $\d$: $\omega\to0$ requires $\d\to0$ {\it and} $C\to\infty$.

\begin{lemma}\label{lemElev}
There exists an absolute constant $K_4$ --- the constant of the error term \eqref{eqOmega} --- with the following property. Fix $\eta_0\in(0,\tfrac12]$ and $C\geq C_1(\eta_0)$, with $C_1$ as in Lemma \ref{lemPair}; then there is a threshold $\d_1(C,\eta_0)>0$ such that the following holds for all $\d\leq\d_1(C,\eta_0)$.

Let $s\in\RR(T,\eta_0,C,\d)$ and $t\geq T$, and suppose that the pair $(s,t)$ is not $\d$-shallow. Then the quantity $\omega=\omega(\d,\eta_0,C)$ of \eqref{eqOmega} satisfies $\omega\leq1$, and:

1) if $E(t,\cdot)$ has no zeros in $Q(s,2C/t)$, then $|u(t,s)-A(s)|\leq\omega$;

2) if $E(t,\cdot)$ has zeros in $Q(s,2C/t)$ and $v\leq2C$ is the smallest of their depths, then
$$u(t,s)\ \leq\ A(s)+\frac12\log\coth v+\omega\ \leq\ A(s)+\frac12\log\left(1+\frac1v\right)+\omega ;$$

3) if, in the setting of part 2), $\tanh v\geq\eta_0$, then also
$$u(t,s)\ \geq\ A(s)+\frac12\log\tanh v-\omega .$$
\end{lemma}

\begin{proof}
1) By the exchanged version of Lemma \ref{lemPair} with $c=C$, $\ti E(t,\cdot)$ has no zeros in $Q(s,C/t)$: a zero of $\ti E$ there would produce a zero of $E$ in $Q(s,2C/t)$. Hence the exponential models of Proposition \ref{propU3}, 2) apply to both families, at $c=2C$ and $c=C$ respectively, and at the point $s$, with $\epsilon_C$ of \eqref{eqEpsU3},
$$\left|E(t,s)-\frac{\a}{\sqrt{w}}e^{-its}\right|\leq\epsilon_C,\qquad
\left|\ti E(t,s)-\frac{\ti\a}{\sqrt{\ti w}}e^{-its}\right|\leq\epsilon_C .$$
Writing $x=w^{-1/2}$, $y=\ti w^{-1/2}$ and using \eqref{eqAid},
$$e^{2A(s)}-\frac{\epsilon_C(x+y)}2\ \leq\ |a(t,s)|^2\ \leq\ e^{2A(s)}+\frac{\epsilon_C(x+y)}2+\frac{\epsilon_C^2}2 .$$
Since $e^{2A}=\frac14(x^2+y^2+2)\geq\frac18(x+y)^2$ and $x+y\geq2$ by \eqref{eqxy2}, the relative deviation is at most $8\epsilon_C$ (we may assume $\epsilon_C\leq1$: otherwise the claim is weaker than $\omega\leq1$), so $|u(t,s)-A(s)|\leq8\epsilon_C=8K_{11}\left(\d+e^{-2(C-1)}\right)\eta_0^{-1}\leq\omega$ for $K_4\geq8e^2K_{11}$; the threshold $\d_1$, together with $C\geq C_1(\eta_0)$ after enlarging $K_{12}$ if needed, guarantees $\omega\leq1$.

2) Let $\zeta$ be a zero of the smallest depth $v$. By Lemma \ref{lemRep}, using $|\sin[t(s-\zeta)]|\leq\cosh v$ and $\g(v)\cosh v=\sqrt{\coth v}$, and $\g(v)\leq\sqrt{\coth v}$,
$$|E(t,s)|\ \leq\ \frac{\sqrt{\coth v}}{\sqrt{w}}\left(1+\vartheta\sqrt w\right).$$
By Lemma \ref{lemPair} with $c=2C$, $\ti E$ has a matched zero of depth $\ti v$ with $\sinh[2\ti v]=(1\pm K_6\vartheta\eta_0^{-3})\sinh[2v]$; since the relative sensitivity of $\coth^2$ with respect to $\sinh[2\cdot]$ is at most $2$, $\coth\ti v\leq(1+K\vartheta\eta_0^{-3})\coth v$, and the Dirichlet part of Lemma \ref{lemRep} gives
$$|\ti E(t,s)|\ \leq\ \frac{\sqrt{\coth v}}{\sqrt{\ti w}}\left(1+K\vartheta\eta_0^{-1}\right).$$
Inserting the two bounds into \eqref{eqAid}, using $\coth v\geq1$ and absorbing the cross terms by \eqref{eqxy2} as in part 1),
$$|a(t,s)|^2\ \leq\ \coth v\ e^{2A(s)}\left(1+K\vartheta\eta_0^{-3}\right),$$
and $\coth v\leq1+1/v$ gives the statement, with $\omega$ absorbing $K\vartheta\eta_0^{-3}\leq K K_5\d^{1/2} e^{16C}\eta_0^{-9/2}$.

3) We use the bound $|\sin(\theta+iv)|\geq\sinh v$ at real $\theta$, $\g(v)\sinh v=\sqrt{\tanh v}$, $\g^2(v)\sinh v=1/\cosh v\leq1$, and the elementary inequality $(\max(0,p-q))^2\geq p^2-2pq$: by Lemma \ref{lemRep},
$$|E(t,s)|^2\ \geq\ \frac{\tanh v}{w}-2\vartheta\,\frac{1}{\sqrt w}\ \geq\ \frac{\tanh v}{w}-2\vartheta\,\eta_0^{-1/2},$$
and similarly for $\ti E$ with $\tanh\ti v\geq(1-K\vartheta\eta_0^{-3})\tanh v$. Hence, by \eqref{eqAid} and $\tanh v\leq1$,
$$|a(t,s)|^2\ \geq\ \tanh v\cdot e^{2A(s)}-K\vartheta\,\eta_0^{-3}.$$
Since $e^{2A}\geq1$ and $\tanh v\geq\eta_0$, the subtracted term is at most $K\vartheta\eta_0^{-4}\tanh v\,e^{2A}$, and
$$u(t,s)\ \geq\ A(s)+\frac12\log\tanh v+\frac12\log\left(1-K\vartheta\eta_0^{-4}\right)\ \geq\ A(s)+\frac12\log\tanh v-\omega$$
after enlarging $K_4$.
\end{proof}

{\bf The grid.} For $\rho>0$ and $T\geq1$ put $t_0=T/32$ and
\begin{equation}
t_{j+1}=\min\left(2t_j,\ \inf\left\{t>t_j:\ \int_{t_j}^t|f|=\rho\right\}\right).
\label{eqGrid}
\end{equation}
Then $t_j\uparrow\infty$, each step carries $|f|$-mass at most $\rho$, and if two consecutive grid points lie in $[\tau/2,\tau)$, the step between them is not a doubling step and therefore carries mass exactly $\rho$; hence the octave counting function satisfies
\begin{equation}
N(\tau):=\#\{j:\ t_j\in[\tau/2,\tau)\}\ \leq\ \frac1\rho\int_{\tau/2}^\tau|f|+1\ \leq\ \frac{2M'}\rho+1,\qquad\tau>0,
\label{eqNgrid}
\end{equation}
with $M'$ from Corollary \ref{corLac}. Consequently,
\begin{equation}
\sum_j\|f\|^2_{L^2((t_j,2t_j))}=\int_{T/32}^\infty f^2(\tau)\,N(\tau)\,d\tau\ \leq\ \left(\frac{2M'}\rho+1\right)\|f\|^2_{L^2((T/32,\infty))}.
\label{eqGridMass}
\end{equation}

{\bf The shallow layer.} It remains to control the shallow pairs excluded from the hypotheses of Sections \ref{secRep}--\ref{secTE}. The mechanism, suggested by the dwell-time phenomenon of \cite{Scatter}, is that a resonance cannot become shallow quickly: descending into the layer requires $|f|$-mass, so the descent must pass, slowly, through the top of the layer --- where the representations are valid --- and its first entry there is controlled by Proposition \ref{propFirstEntry}. We first estimate the $|f|$-mass required for a descent.

\begin{lemma}[descent estimate]\label{lemDescent}
There is an absolute constant $K_6'$ with the following property. Let $s\in\RR(T,\eta_0,C,\d)$, let $T\leq t_1<t_2$, and let $\zeta(\tau)$ be a continuous family of zeros of $E(\tau,\cdot)$ lying in $Q(s,16C/\tau)$ for $\tau\in[t_1,t_2]$, with depths $v(\tau)\leq1$. Then
\begin{equation}
\sinh v(t_2)\ \geq\ \tfrac16\,e^{-2\int_{t_1}^{t_2}|f|}\,\sinh v(t_1)\ -\ 2\,\d\,\eta_0^{-1} .
\label{eqDescent}
\end{equation}
The same holds for the Dirichlet family.
\end{lemma}

\begin{proof}
Write $x_0=\Re\zeta(t_2)$, $W=\R\cap Q(s,16C/t_1)$, $\int=\int_{t_1}^{t_2}|f|$, and let $\b_i=\b(t_i)$ be the normalizations of Lemma \ref{lemBeta} at $\zeta(t_i)$. By \eqref{eqMappE}, which holds at every depth, for $p\in W$ and $i=1,2$,
\begin{equation}
\Big|\,|\b_i|\,|E(t_i,p)|-\frac{\left|\sin[t_i(p-\zeta(t_i))]\right|}{w}\,\Big|\ \leq\ \d .
\label{eqPattern}
\end{equation}
Four applications of \eqref{eqPattern} and of the value rigidity $|\Delta\log|E(\tau,x)||\leq\int$ (Remark \ref{remFloor}) give the claim. First, at $p=x_0$ and $i=2$, since $|\sin[t_2(x_0-\zeta(t_2))]|=\sinh v(t_2)$,
$$|\b_2|\,|E(t_2,x_0)|\ \leq\ \frac{\sinh v(t_2)+\d w}{w}.$$
Second, at $p=x_0$ and $i=1$, since $|\sin[t_1(x_0-\zeta(t_1))]|\geq\sinh v(t_1)$,
$$|\b_1|\,|E(t_1,x_0)|\ \geq\ \frac{\sinh v(t_1)-\d w}{w}.$$
Third, taking $p=p^*$, the point of $W$ where $|\sin[t_2(p-\zeta(t_2))]|$ is maximal (that maximum lies in $[1,\cosh1]$ since $v(t_2)\leq1$ and $W$ is longer than the period), $|\b_2|\sup_W|E(t_2,\cdot)|\geq|\b_2||E(t_2,p^*)|\geq(1-\d w)/w\geq1/(2w)$; and similarly, at $i=1$, $|\b_1|\sup_W|E(t_1,\cdot)|\leq(\cosh 1+\d w)/w\leq3/w$. Since the value rigidity gives $\sup_W|E(t_2,\cdot)|\leq e^{\int}\sup_W|E(t_1,\cdot)|$ pointwise on $W$, the last two bounds yield $|\b_2|/|\b_1|\geq\tfrac16e^{-\int}$. Finally, the rigidity at the fixed point $x_0$ gives $|E(t_2,x_0)|\geq e^{-\int}|E(t_1,x_0)|$, and chaining the four displays,
$$\frac{\sinh v(t_2)+\d w}{w}\ \geq\ |\b_2|\,e^{-\int}\,|E(t_1,x_0)|\ \geq\ \tfrac16\,e^{-2\int}\,\frac{\sinh v(t_1)-\d w}{w},$$
which is \eqref{eqDescent} after $w\leq\eta_0^{-1}$.
\end{proof}

Recall from Remark \ref{remFloor} that, under \eqref{eqDMR}, every zero of $E(t,\cdot)$ or $\ti E(t,\cdot)$, at every real location, has depth at least $c\,e^{-2(M'+M)}t^{-2M/\ln2}$. Define
\begin{equation}
T_\d\ :=\ \left(\frac{c\,e^{-2(M'+M)}}{\uv}\right)^{\ln2/(2M)}\qquad(M>0;\ T_\d:=\infty\text{ if }M=0),
\label{eqTdelta}
\end{equation}
so that no $\d$-shallow pairs exist at times $t\leq T_\d$, and $T_\d\to\infty$ as $\d\to0$ for fixed $M,M'$.

\begin{lemma}[first entry into the layer]\label{lemShallow}
There exist absolute constants $K$ and $K_6''$ with the following property. Fix $\eta_0\in(0,\tfrac12]$, $C\geq C_1(\eta_0)$ and $M'\geq0$; then there is a threshold $\d_5(C,\eta_0,M')>0$ such that the following holds for all $\d\leq\d_5$.

Let $T\geq1$, let the grid \eqref{eqGrid} run with a step $\rho\in(0,1]$, and put $T_*:=\max(T,T_\d)$, with $T_\d$ defined in \eqref{eqTdelta}. Then
\begin{equation}
\begin{gathered}
\left|\left\{s\in\RR(T,\eta_0,C,\d):\ (s,t)\text{ is $\d$-shallow for some }t\geq T_*\right\}\right|\ \leq\\
\leq\ K\left(\frac{M'}\rho+1\right)\frac{e^{8C}\,\eta_0^{-2}}{\log(1/\d)}\ \|f\|^2_{L^2((T^\sharp,\,\infty))},\qquad T^\sharp:=\max(T,T_\d)/32 .
\end{gathered}
\label{eqShallowSet}
\end{equation}
Moreover, if in the left-hand side the shallow zero is additionally required to have depth at most $v\leq\uv$, the factor $1/\log(1/\d)$ improves to $1/\Phi(v^\sharp,C)$ with $v^\sharp:=v\,e^{K_6''(M'+C+1)}$, provided $v^\sharp\leq C$.
\end{lemma}

\begin{proof}
Let $s$ belong to the set in \eqref{eqShallowSet} and let $t^\flat=t^\flat(s)\geq T_*$ be the first time at which $(s,t)$ is $\d$-shallow; by the choice of $T_\d$ and continuity, $t^\flat$ is well defined, and at $t^\flat$ some zero $\zeta^0$ of $E(t^\flat,\cdot)$ or of $\ti E(t^\flat,\cdot)$ in $Q(s,16C/t^\flat)$ has depth exactly $\uv$, while for $t<t^\flat$ no zero of either family in $Q(s,16C/t)$ has depth below $\uv$. We treat the case of a zero of $E$; the Dirichlet case is identical --- the argument below uses only the pattern \eqref{eqPattern} and the value rigidity, both of which hold for the Dirichlet family, and its conclusion, a zero of $\ti E(t_1,\cdot)$ in $Q(s,2C/t_1)$ of depth between $\uv$ and $v^\sharp$, is converted into a zero of $E(t_1,\cdot)$ in $Q(s,4C/t_1)$ of comparable depth by the exchanged version of Lemma \ref{lemPair} (the pair $(s,t_1)$ is not $\d$-shallow since $t_1<t^\flat$, and the depths match by \eqref{eqDepthMatch}); Proposition \ref{propFirstEntry} is then applied with $2v^\sharp$ in place of $v^\sharp$.

Let $t_1$ be the largest grid point with $t_1\leq t^\flat/16$; it exists, since $t_0=T/32\leq T_*/16\leq t^\flat/16$. We claim that $E(t_1,\cdot)$ has a zero in $Q(s,2C/t_1)$ of depth between $\uv$ and $v^\sharp:=\uv e^{K_6''(M'+C+1)}$. Indeed, write $x_0=\Re\zeta^0\in Q(s,16C/t^\flat)$; then $t_1|x_0-s|\leq(t_1/t^\flat)\cdot16C\leq C$. The interval $(t_1,t^\flat)$ spans at most five octaves, at most one of which may intersect $(0,1)$, so the mass $\int_{t_1}^{t^\flat}|f|$ is at most $5M+M'\leq K(M'+1)$. If $E(t_1,\cdot)$ had no zeros in $Q(s,8C/t_1)$, Proposition \ref{propU3}, 2) with $c=8C$ would give $|E(t_1,x_0)|\geq1/\sqrt w-\epsilon_{8C}\geq\tfrac12\sqrt{\eta_0}$ (the term $e^{-16C+2}\eta_0^{-1}$ of $\epsilon_{8C}$ is at most $\tfrac14\sqrt{\eta_0}$ by $C\geq C_1(\eta_0)$, and the term $K_{11}\d\eta_0^{-1}$ by $\d\leq\d_5$), while the pattern \eqref{eqPattern} at $t^\flat$ (at $p=x_0$, where $|\sin[t^\flat(x_0-\zeta^0)]|=\sinh\uv$) together with \eqref{eqBetaAdd} gives $|E(t^\flat,x_0)|\leq K\sqrt{(\uv+\d\eta_0^{-1})\eta_0^{-1}}$, whence, by the value rigidity, $|E(t_1,x_0)|\leq e^{K(M'+1)}|E(t^\flat,x_0)|$ --- a contradiction for $\d\leq\d_5$. Hence $E(t_1,\cdot)$ has zeros in $Q(s,8C/t_1)$. Let $\zeta'$ be one nearest to $x_0$, $v'$ its depth and $\b'$ its normalization \eqref{eqBetaAdd}. By \eqref{eqPattern} at $t_1$,
$$\begin{gathered}\left|\sin[t_1(x_0-\zeta')]\right|\ \leq\ w\,|\b'|\,|E(t_1,x_0)|+\d w\ \leq\\ \leq\ K\sqrt{\sinh[2v']+\d\eta_0^{-1}}\cdot e^{K(M'+1)}\sqrt{(\uv+\d\eta_0^{-1})\,\eta_0^{-1}}+\d\eta_0^{-1},\end{gathered}$$
using $|\b'|\leq K\sqrt{(\sinh[2v']+\d\eta_0^{-1})/w}$ from \eqref{eqBetaAdd} and $w\leq\eta_0^{-1}$. Since $|\sin(\theta+iv)|^2=\sin^2\theta+\sinh^2v$ for real $\theta,v$, the left-hand side dominates both $\sinh v'$ and $|\sin[t_1(x_0-\Re\zeta')]|$, and the display first gives $\sinh v'\leq K e^{K(M'+1)}\eta_0^{-2}\uv$ (absorbing the $\d$-terms for $\d\leq\d_5$), and then, returning to it with the depth bound in hand, $|\sin[t_1(x_0-\Re\zeta')]|\leq Ke^{K(M'+1)}\eta_0^{-2}\uv$ as well, so that $x_0$ lies within $Ke^{K(M'+1)}\eta_0^{-2}\uv/t_1$ of the real part of a zero of the pattern, and hence of a zero of $E(t_1,\cdot)$. Since $C\geq C_1(\eta_0)$, the powers of $\eta_0^{-1}$ are absorbed into $e^{KC}$, and both the distance $t_1|x_0-\Re\zeta'|$ and the depth $v'$ are at most $\uv e^{K_6''(M'+C+1)}=v^\sharp$; in particular $\zeta'\in Q(s,2C/t_1)$, and, since $t_1<t^\flat$, its depth is at least $\uv$.

Thus every $s$ of the set in \eqref{eqShallowSet} exhibits, at some grid point $t_1\geq T^\sharp$ ($t_1\geq t^\flat/32\geq T_*/32=T^\sharp$), the event of Proposition \ref{propFirstEntry} with $v=v^\sharp$ (or $v=2v^\sharp$ in the Dirichlet case); the pairs $(s,t_1)$ and $(s,2t_1)$ are not $\d$-shallow since $2t_1\leq t^\flat/8<t^\flat$. Summing \eqref{eqFirstEntry} over the grid points $t_j\geq T^\sharp$ and using \eqref{eqGridMass},
$$\left|\left\{\cdots\right\}\right|\ \leq\ \frac{2^{17}\pi e^{8C}}{\eta_0^{2}\,\Phi(v^\sharp,C)}\left(\frac{2M'}\rho+1\right)\|f\|^2_{L^2((T^\sharp,\infty))} .$$
Finally,
$$\Phi(v^\sharp,C)=\frac12\log\left(1+\frac{2^{-13}e^{-8C}}{v^\sharp}\right)\ \geq\ c\left(\log\frac1\d-K(M'+C+1)\right)\ \geq\ c'\log\frac1\d$$
for $\d\leq\d_5(C,\eta_0,M')$, which gives \eqref{eqShallowSet}; the graded version is the same estimate with the first-entry threshold $\uv$ replaced by $v$.
\end{proof}

{\bf The main estimates.} Let us collect the objects entering the statements below. The regular set $\RR(T,\eta_0,C,\d)$, defined in Section \ref{secU} through the maximal conditions (M1)--(M3), consists of points $s\in\R$ with $\min\left(w(s),\ti w(s)\right)\geq\eta_0$, and on it condition (U2) holds with accuracy $\d$ for all $t\geq T$ \eqref{eqRtoU}, while the resonance-free exponential models hold with the accuracies $\epsilon_c$ of \eqref{eqEpsU3}, by Proposition \ref{propU3}; for fixed $\eta_0,C,\d$ these sets increase with $T$ and exhaust almost all of $\{\min(w,\ti w)\geq\eta_0\}$. The layer depth $\uv\approx\tfrac12\d^{1/2}$ is defined in \eqref{eqLayer}, and $T_\d$, the time before which no resonance can be $\d$-shallow, in \eqref{eqTdelta}. Further, $C_0\geq4\pi$ and $K_{12}$ are the absolute constants fixed in Lemmas \ref{lemPair} and \ref{lemSweep}, and $C_1(\eta_0)=C_0+K_{12}(1+\log(1/\eta_0))$ is the least admissible scale;
$$A(s)=\log\left(\frac12\sqrt{\frac1{w(s)}+\frac1{\ti w(s)}+2}\right)$$
is the limiting value of $\log|a(t,s)|$ introduced at the beginning of this section; $\omega=K_4(\d^{1/2} e^{16C}\eta_0^{-6}+e^{-2C}\eta_0^{-1})$ is the additive error \eqref{eqOmega} of the elevation lemma, made arbitrarily small by decreasing $\d$ {\it and} increasing $C$; $M=\sup_n\int_{2^n}^{2^{n+1}}|f|$ is the dyadic mass of \eqref{eqDMR} and $M'=\max\left(M,\int_0^1|f|\right)$ is its adjusted version from Corollary \ref{corLac}.

Finally, we adapt the regular sets to the deviation level: besides the spectral-gap parameter $\eta_0$, the set on which the estimate at the level $\lambda$ is proved depends on $\lambda$ itself. For $\eta_0\in(0,\tfrac12]$, $\lambda>0$ and $M\geq0$ put
\begin{equation}
C(\eta_0,\lambda):=C_0+K_{13}\left(1+\log(1/\eta_0)+\log\left(1/\min(\lambda,1)\right)\right),
\label{eqCanon}
\end{equation}
with an absolute constant $K_{13}\geq K_{12}$ chosen large enough for the requirements of the proof below, let $\d(\eta_0,\lambda,M')>0$ be the explicit threshold fixed in that proof --- it does not exceed $e^{-\lambda}$ or the thresholds of the preceding lemmas at the scale \eqref{eqCanon} --- and define
\begin{equation}
\RR(T,\eta_0,\lambda):=\RR\left(T,\eta_0,C(\eta_0,\lambda),\d(\eta_0,\lambda,M')\right).
\label{eqCanonSet}
\end{equation}
The sets \eqref{eqCanonSet} are defined through the maximal functions only. Both a small and a large $\lambda$ make them more restrictive: a small $\lambda$ raises the scale, a large $\lambda$ lowers the threshold $\d$. For every fixed $(\eta_0,\lambda)$ they increase in $T$ and exhaust almost all of $\{\min(w,\ti w)\geq\eta_0\}$; at a given point $s$, a more restrictive choice requires a larger $T$.

\begin{theorem}[Maximal fluctuation estimate on regular sets]\label{thmMaxR}
Let $f\in L^2(\R_+)$ be real and satisfy \eqref{eqDMR}. There exist absolute constants $K$ and $N$ such that for all $\eta_0\in(0,\tfrac12]$, $\lambda>0$, $M\geq0$ and $T\geq1$,
\begin{equation}
\begin{gathered}
\left|\left\{s\in\RR(T,\eta_0,\lambda):\ \sup_{t\geq T}\big|\log|a(t,s)|-A(s)\big|>\lambda\right\}\right|\ \leq\\
\leq\ K\,(M'+1)\,\left(\eta_0\min(\lambda,1)\right)^{-N}\ \frac{\|f\|^2_{L^2((T/32,\infty))}}{\lambda}.
\end{gathered}
\label{eqMaxT}
\end{equation}
\end{theorem}

\begin{remark}\label{remOneDial}
For $\lambda\geq1$ the estimate is the weak-type bound with the constant $K(M'+1)\eta_0^{-N}$, and the set depends on $\lambda$ only through the requirement $\d\leq e^{-\lambda}$: on a fixed set $\RR(T,\eta_0,1)$ the same bound holds for all $1\leq\lambda\leq\log(1/\d)$. For $\lambda<1$ the factor $\lambda^{-N}$ comes from the limited accuracy of the models: by the intrinsic accuracy limit of Proposition \ref{propU3} (see \eqref{eqEpsU3} and the remark following it), the resonance-free models at a scale $C$ carry the irreducible error $\gtrsim e^{-2C}\eta_0^{-1}$, so deviations of size $\lambda$ are detectable only at scales $C\gtrsim\log(1/\eta_0\lambda)$, and the constants of Sections \ref{secTE}--\ref{secMax} are exponential in the scale.
\end{remark}

The theorem follows from the parametric estimate that the machinery of Sections \ref{secRep}--\ref{secMax} gives directly; the reader interested in the interplay of the parameters may use it in place of the canonical choices \eqref{eqCanon}, \eqref{eqCanonSet}.

\begin{proposition}[parametric maximal estimate]\label{propMaxPar}
Let $f\in L^2(\R_+)$ be real and satisfy \eqref{eqDMR}. There exists an absolute constant $K$ with the following property. Fix $\eta_0\in(0,\tfrac12]$, $C\geq C_1(\eta_0)$ and $M'\geq0$; then there is a threshold $\d_2(C,\eta_0,M')>0$ such that the following holds for all $\d\leq\d_2(C,\eta_0,M')$.

Let $T\geq1$, let $\rho\in(0,1]$ be the step of the grid \eqref{eqGrid}, and let $\lambda>0$. Then
\begin{equation}
\begin{gathered}
\left|\left\{s\in\RR(T,\eta_0,C,\d):\ \sup_{t\geq T}\big|\log|a(t,s)|-A(s)\big|>\lambda+\rho+2\omega\right\}\right|\ \leq\\
\leq\ K\,(M'+1)\,e^{16C}(C+1)\,\eta_0^{-2}\left[\frac{\|f\|^2_{L^2((T/32,\infty))}}{\rho\,\lambda}
+\frac{\|f\|^2_{L^2((T^\sharp,\infty))}}{\rho\,\log(1/\d)}\right],
\end{gathered}
\label{eqMaxR}
\end{equation}
where $\omega=\omega(\d,\eta_0,C)$ is the elevation error \eqref{eqOmega} and $T^\sharp=\max(T,T_\d)/32$, as in Lemma \ref{lemShallow}. In particular, for $\lambda\leq\log(1/\d)$ the second term is dominated by the first with the tail $\|f\|^2_{L^2((T^\sharp,\infty))}$ in place of $\|f\|^2_{L^2((T,\infty))}$; and since $T_\d\to\infty$ as $\d\to0$, the second term vanishes in every limit in which $\d\to0$.
\end{proposition}

The additive term $\rho+2\omega$ and the two terms on the right-hand side of \eqref{eqMaxR} reflect the three mechanisms of the proof: the grid step $\rho$ accounts for the Lipschitz oscillation within a step, the elevation error $\omega$ collects the model errors, and the second term bounds the contribution of the shallow layer; the latter vanishes in every limit in which $\d\to0$, since $T_\d\to\infty$.

\begin{proof}
We take $\d_2\leq\min(\d_0,\d_1,\d_5)(C,\eta_0,M')$, so that Lemmas \ref{lemPair}--\ref{lemShallow} and Proposition \ref{propFirstEntry} apply. Let $H$ denote the set of Lemma \ref{lemShallow}; its measure is bounded by the second term of \eqref{eqMaxR}, so it suffices to estimate the set on the left with $H$ removed. For $s\notin H$ no pair $(s,t)$, $t\geq T$, is $\d$-shallow: for $t\leq T_\d$ this holds by the choice \eqref{eqTdelta} of $T_\d$, and for $t\geq T_*$ by the definition of $H$; in particular the hypotheses of Lemmas \ref{lemPair}--\ref{lemElev} and of Proposition \ref{propFirstEntry} hold at every $t\geq T$. Take the grid \eqref{eqGrid} with the step $\rho$. By Lemma \ref{lemLip}, $|u(t,s)-u(t_j,s)|\leq\rho$ for $t\in[t_j,t_{j+1}]$; hence if the supremum in \eqref{eqMaxR} exceeds $\lambda+\rho+2\omega$ then $|u(t_j,s)-A(s)|>\lambda+2\omega$ for some $j$. By Lemma \ref{lemElev}, part 1) is then impossible at $(s,t_j)$, so $E(t_j,\cdot)$ has zeros in $Q(s,2C/t_j)$; let $v$ be the smallest of their depths (as $s\notin H$, $v\geq\uv$). We claim that in both cases
\begin{equation}
v\ \leq\ v_*(\lambda):=\min\left(\left(e^{2\lambda}-1\right)^{-1},\ C\right).
\label{eqDepthForced}
\end{equation}
Observe first that $v<C$. Indeed, if $v\geq C$, then parts 2) and 3) of Lemma \ref{lemElev} (the latter applies since $\tanh v\geq\tanh C\geq\eta_0$) give $|u(t_j,s)-A(s)|\leq\frac12\log\coth C+\omega\leq2e^{-2C}+\omega\leq2\omega$, contradicting the deviation $>\lambda+2\omega$; here $2e^{-2C}\leq\omega$ by \eqref{eqOmega}. If $u(t_j,s)-A(s)>\lambda+2\omega$, part 2) of Lemma \ref{lemElev} gives $\frac12\log(1+1/v)>\lambda$, which together with $v<C$ is \eqref{eqDepthForced}. Suppose $u(t_j,s)-A(s)<-\lambda-2\omega$. Since $u\geq0$, this is possible only when $\lambda<A(s)\leq\frac12\log(1/\eta_0)$, i.e. $\eta_0<e^{-2\lambda}$. If $\tanh v\geq\eta_0$, part 3) gives $\tanh v<e^{-2\lambda}$; if $\tanh v<\eta_0$, then $\tanh v<e^{-2\lambda}$ holds trivially. In both subcases
$$v\ <\ \arctanh\left(e^{-2\lambda}\right)\ \leq\ \frac{e^{-2\lambda}}{1-e^{-2\lambda}}\ =\ \left(e^{2\lambda}-1\right)^{-1},$$
where the middle inequality is $\log y\leq y-1$ applied to $\arctanh x=\frac12\log\frac{1+x}{1-x}\leq\frac x{1-x}$. This proves \eqref{eqDepthForced}: a large deviation of either sign forces a shallow resonance.

Therefore the set in \eqref{eqMaxR}, with $H$ removed, is contained in the union over $j$ of the sets of Proposition \ref{propFirstEntry} with $v=\max(v_*(\lambda),\uv)$; since $\Phi(\max(v_*,\uv),C)\geq\min\left(\Phi(v_*(\lambda),C),\,c\log(1/\d)\right)$ and the $\log(1/\d)$-branch is again absorbed by the second term of \eqref{eqMaxR}, we may use $\Phi(v_*(\lambda),C)$. By \eqref{eqGridMass},
$$\left|\left\{\cdots\right\}\right|\ \leq\ \frac{2^{17}\pi e^{8C}}{\eta_0^2\,\Phi(v_*(\lambda),C)}\left(\frac{2M'}\rho+1\right)\|f\|^2_{L^2((T/32,\infty))}.$$
It remains to show that
\begin{equation}
\Phi(v_*(\lambda),C)\ \geq\ 2^{-17}\,\frac{e^{-8C}}{C+1}\;\lambda\qquad\text{ for all }\lambda>0.
\label{eqPhiLow}
\end{equation}
First, $1/v_*(\lambda)\geq2\lambda$ in all cases: if $v_*(\lambda)=(e^{2\lambda}-1)^{-1}$ this is $e^{2\lambda}-1\geq2\lambda$, and if $v_*(\lambda)=C<(e^{2\lambda}-1)^{-1}$ then $2\lambda\leq e^{2\lambda}-1<1/C$. Put $x=2^{-13}e^{-8C}/v_*(\lambda)$. If $x\leq1$, then $\Phi=\frac12\log(1+x)\geq x/4\geq2^{-15}e^{-8C}\cdot2\lambda$, which is stronger than \eqref{eqPhiLow}. If $x>1$, then $e^{2\lambda}-1\geq1/v_*(\lambda)>2^{13}e^{8C}$, so $\lambda\geq4C$; in this range $\Phi\geq\frac12\log x\geq\lambda-4C-5$, which is $\geq\lambda/2$ for $\lambda\geq8C+10$, while for $4C\leq\lambda\leq8C+10$ we still have $\Phi\geq\frac12\log2\geq\frac{\log2}2\cdot\frac\lambda{8C+10}$; both bounds imply \eqref{eqPhiLow}. Combining the two displays gives \eqref{eqMaxR}.
\end{proof}

\begin{proof}[Proof of Theorem \ref{thmMaxR}]
Set $C=C(\eta_0,\lambda)$ as in \eqref{eqCanon}, $\rho=\min(\lambda,1)/8$, and
$$\d(\eta_0,\lambda,M'):=\min\left(\d_2\left(C,\eta_0,M'\right),\ e^{-\lambda},\ c\min(\lambda,1)^2\,\eta_0^{12}\,e^{-32C}\right)$$
with a small absolute constant $c$. Then $2\omega\leq3\min(\lambda,1)/8$: the $\d$-term of $2\omega$ is $2K_4\d^{1/2}e^{16C}\eta_0^{-6}\leq2K_4c^{1/2}\min(\lambda,1)\leq\min(\lambda,1)/4$ for $c$ small, and the $\d$-independent term is $2K_4e^{-2C}\eta_0^{-1}\leq2K_4e^{-2C_0-2K_{13}}\left(\eta_0\min(\lambda,1)\right)^{2K_{13}}\eta_0^{-1}\leq\min(\lambda,1)/8$ by the choice of $K_{13}$. Hence
$$\frac\lambda2+\rho+2\omega\ \leq\ \frac\lambda2+\frac{\min(\lambda,1)}2\ \leq\ \lambda,$$
so the set in \eqref{eqMaxT} is contained in the set of \eqref{eqMaxR} taken at the level $\lambda/2$. Since $\log(1/\d)\geq\lambda$, the second term of \eqref{eqMaxR} is at most the first (with $\|f\|^2_{L^2((T^\sharp,\infty))}\leq\|f\|^2_{L^2((T/32,\infty))}$), and
$$e^{16C}\left(C+1\right)\,\frac{\eta_0^{-2}}{\rho}\ \leq\ K\,\left(\eta_0\min(\lambda,1)\right)^{-N}$$
with an absolute $N$, by \eqref{eqCanon}. Collecting the factors gives \eqref{eqMaxT}.
\end{proof}

\begin{corollary}\label{corConv}
Let $f\in L^2(\R_+)$ be real and satisfy \eqref{eqDMR}. Then for a.e. $s\in\R$,
$$\lim_{t\to\infty}\log|a(t,s)|\ =\ A(s) .$$
\end{corollary}

\begin{proof}
Fix $\eta_0$ and $\lambda\in(0,1]$. By Theorem \ref{thmMaxR}, letting $T\to\infty$ and recalling that the sets $\RR(T,\eta_0,\lambda)$ increase to a set of full measure in $\{\min(w,\ti w)\geq\eta_0\}$, we conclude that for a.e. $s$ in the latter set, $\limsup_{t\to\infty}|u(t,s)-A(s)|\leq\lambda$. Taking countable sequences $\lambda\downarrow0$ and then $\eta_0\downarrow0$ (recall that $\min(w,\ti w)>0$ a.e.), we obtain the convergence.
\end{proof}

Recall that the limit function satisfies the Parseval identity $\int_\R A(s)\,ds=\frac\pi2\|f\|^2_{L^2(\R_+)}$ (see for instance \cite{Den}); by Chebyshev's inequality,
\begin{equation}
|\{s\in\R:\ A(s)>\lambda\}|\ \leq\ \frac{\pi\|f\|_2^2}{2\lambda},\qquad\lambda>0.
\label{eqAweak}
\end{equation}

Combining Theorem \ref{thmMaxR} with \eqref{eqAweak} and Lemma \ref{lemLip} gives a weak-type maximal estimate for $u^*$ on the regular sets: for every $\lambda>0$,
\begin{equation}
\begin{gathered}
\left|\left\{s\in\RR(T,\eta_0,\lambda):\ u^*(s)>2\lambda+\int_0^T|f|\right\}\right|\ \leq\\
\leq\ K\,(M'+1)\,\left(\eta_0\min(\lambda,1)\right)^{-N}\ \frac{\|f\|^2_{2}}{\lambda},
\end{gathered}
\label{eqMaxRstar}
\end{equation}
since $u^*\leq\max\left(\int_0^T|f|,\ \sup_{t\geq T}u(t,\cdot)\right)$ and the set $\{A>\lambda\}$ is estimated by \eqref{eqAweak}.

\begin{remark}
The power $N$ and the thresholds $\d(\eta_0,\lambda,M')$ are what the proofs give; no attempt has been made to optimize them.
\end{remark}

\begin{remark}\label{remExtract}
The only non-quantitative ingredient of the arguments above is the rate at which the kernel-approximation error in (U2) tends to zero, which for a general $f\in L^2$ admits no modulus; this is the reason the estimates are localized to the regular sets $\RR$. The pointwise convergence of the transform $\nlhat f_t$ itself, which requires control of the relative phases of the scattering data and different methods, will be addressed elsewhere; Corollary \ref{corConv} recovers its modulus layer, $|\nlhat f_t|\to|\nlhat f|$ a.e., through the identity $u=-\frac12\log(1-|\nlhat f_t|^2)$.
\end{remark}

\begin{remark} The restriction of Theorem \ref{thmMaxR} to the modulus $|a|$ is essential. As shown in \cite{Den2}, there exist potentials for which $\arg a(t,s)$ diverges as $t\to\infty$ at every $s$, while $|a(t,s)|$ and $|b(t,s)|$ converge and the maximal function of $\log|a|$ remains finite. The phase $\arg a=H(\log|a(t,\cdot)|)(s)$ belongs to the conjugate-function layer of the problem, where only weak-type substitutes may survive; Theorem \ref{thmMaxR} concerns the layer controlled by the Parseval identity \eqref{eqPars}.
\end{remark}

\appendix

\section{Universality conditions from maximal functions}\label{secApp}

In this appendix we prove condition (U2) from the maximal conditions (M1)--(M3) of Section \ref{secU}. The arguments below are quantitative versions of the universality machinery of \cite{Scatter}.

Throughout the appendix $s\in\R$ is fixed and $w=w(s)>0$; we use the Poisson kernel $P_z$, the cones $\G_\s$, the truncated maximal functions $M_h$, $\MM_h$, $\MM^2_h$, the outer functions $G$, $\ti G$ and the conditions (M1)--(M3), whose truncations are at the height $50C/t$, all as introduced in Section \ref{secU}. We will use repeatedly two elementary facts (see \cite{G}, Ch.~I): the truncated non-tangential maximal function $M_h$ of a measure is dominated, on $\R$, by an absolute constant times the Hardy--Littlewood maximal function, and hence satisfies the weak-type $(1,1)$ inequality; and for $z=x'+ia$, $z'=ib$ with $0<a\leq b$,
\begin{equation}
P_{x'+ia}(x)\ \leq\ \frac{2\left((x')^2+b^2\right)}{ab}\ P_{ib}(x)\qquad\text{ for all }x\in\R .
\label{eqAker}
\end{equation}

\begin{theorem}\label{thmAppMain}
There exists an absolute constant $K_8$ with the following property. Let $C\geq1$, $\e\in(0,\tfrac14]$, and suppose that $(s,t)$ satisfies {\rm(M1)--(M3)} with
\begin{equation}
\e_0\ \leq\ \e^{8}\,\min\left(w(s),\ti w(s),1\right)^{8}e^{-K_8C}.
\label{eqAthr}
\end{equation}
Then {\rm(U2)} holds at $(s,t)$ with the given $\e$.
\end{theorem}

The almost-everywhere-eventually property claimed in Section \ref{secU} now takes the following classical form.

\begin{proposition}\label{propAae}
For almost every $s\in\R$,
$$\MM_h\log w\,(s)\to0,\qquad \MM^2_hG(s)\to0,\qquad M_h\mu(s)\to w(s)\qquad\text{ as }h\to0^+,$$
and the same holds for the Dirichlet data. Consequently, for every fixed $C\geq1$ and $\e_0>0$, conditions {\rm(M1)--(M3)} hold at $(s,t)$, for almost every $s$ with $w(s)\ti w(s)\neq0$, for all sufficiently large $t$.
\end{proposition}

\begin{proof}
The first limit holds at a.e.\ point since $\log w\in L^1(\Pi)$: this is the standard consequence of the weak-type bound for the non-tangential maximal function and the density of continuous functions (\cite{G}, Ch.~I). The second is the same statement for $|G-G(s)|^2$: a.e.\ $s$ is an $L^2(\Pi)$-Lebesgue point of $G$ in the non-tangential sense. The third is the non-tangential differentiation theorem for the Poisson-finite measure $\mu=w\,m+\mu_{\rm sing}$. The last claim follows since $50C/t\to0$ as $t\to\infty$.
\end{proof}

The proof of Theorem \ref{thmAppMain} occupies the rest of the appendix. Constants denoted by $K$ are absolute and may change from line to line; the dependence on $\e$, $C$ and the densities is tracked explicitly. It suffices to consider the Neumann family; the Dirichlet case is identical.

\begin{lemma}[Poisson majorization]\label{lemAPS}
Let $x_0\in\R$, $h\geq1$, $t>0$ and $w_0=x_0+ih/t$. Then
$$\frac{|\S(t,w_0,x)|^2}{\S(t,w_0,w_0)}\ \leq\ 8\,P_{x_0+i(h+1)/t}(x)\qquad\text{ for all }x\in\R .$$
\end{lemma}

\begin{proof}
Since $|\sin[t(x-\bar w_0)]|^2=|\sin[t(x-x_0)+ih]|^2\leq\cosh^2h$ and $\S(t,w_0,w_0)=\frac{t\sinh2h}{2\pi h}$,
$$\frac{|\S(t,w_0,x)|^2}{\S(t,w_0,w_0)}\ \leq\ \frac{2h\cosh^2h}{\pi t\,\sinh 2h}\cdot\frac1{(x-x_0)^2+(h/t)^2}\ =\ \frac{h\coth h}{\pi t\left[(x-x_0)^2+(h/t)^2\right]} .$$
Now $h\coth h\leq h+1$ and $(x-x_0)^2+((h+1)/t)^2\leq4\left[(x-x_0)^2+(h/t)^2\right]$ (because $h+1\leq2h$), so
$$\frac{h+1}{\pi t\left[(x-x_0)^2+(h/t)^2\right]}\ \leq\ \frac{4(h+1)/t}{\pi \left[(x-x_0)^2+\frac{(h+1)^2}{t^2}\right]}\ \leq\ 8\,P_{x_0+i(h+1)/t}(x). \qedhere$$
\end{proof}

\begin{lemma}[monotonicity and reflection]\label{lemAMon}
1) If $F\in\PW_t$ has no zeros in $\C_+$, then $e^{ty}|F(x+iy)|$ is non-decreasing in $y\geq0$ for every fixed $x\in\R$.

2) For any $F\in\PW_t$, the reflected function
$$F'(z)=F(z)\prod_{a_j\in\C_+}\frac{z-\bar a_j}{z-a_j},$$
the product taken over the zeros of $F$ in $\C_+$, is entire, belongs to $\PW_t$, has no zeros in $\C_+$, and satisfies $|F'|=|F|$ on $\R$ and $|F'|\geq|F|$ in $\C_+$.

3) Let $\nu$ be a positive Poisson-finite measure, $\l\in\C_+$, and let $F\in\PW_t$, $\|F\|_\nu\leq1$, maximize $|F(\l)|$ among such functions. Then $F$ has no zeros in $\C_+$, and $e^{ty}|F(x+iy)|$ is non-decreasing in $y\geq0$ for every $x$.
\end{lemma}

\begin{proof}
1) Write the Cartwright representation $F=d z^ne^{ikz}\lim_R\prod_{|a_j|<R}(1-z/a_j)=dz^ne^{ikz}G_0(z)$ (see \cite{dBr}). All $a_j$ lie in $\bar\C_-$, so each factor $|1-z/a_j|$, as well as $|z|^n$, is a non-decreasing function of $y=\Im z$ on $y\geq0$ for fixed $x$; hence so is $|z^nG_0(x+iy)|$. It remains to check $k\leq t$. In $\C_+$, $G_0=Oe^{ilz}$ with $O$ outer, and $l\leq0$ since $|G_0(iy)|$ is non-decreasing while $\log|O(iy)|=o(y)$ (as $\log|F|\in L^1(\Pi)$ for Cartwright functions). Since $G_0$ has no zeros in $\C_+$, $G_0^\#=OBe^{ilz}$ in $\C_+$ for a Blaschke product $B$, and $e^{itz}F^\#=\bar d z^nOBe^{i(t-k+l)z}\in H^2(\C_+)$ forces $t-k+l\geq0$, i.e.\ $k\leq t+l\leq t$. Therefore $e^{ty}|F(x+iy)|=|d|\,e^{(t-k)y}\,|z^nG_0(x+iy)|$ is non-decreasing in $y$.

2) The zeros of $F$ in $\C_+$ satisfy the Blaschke condition (Cartwright class), so the product converges; the poles of the product cancel the zeros of $F$, making $F'$ entire, with $|F'|=|F|$ on $\R$ (hence $F'\in L^2(\R)$), $|F'|\geq|F|$ in $\C_+$ (each factor has modulus $\geq1$ there), and the exponential type is unchanged since the product is bounded in each half-plane away from the zeros. Thus $F'\in\PW_t$, and by construction it has no zeros in $\C_+$.

3) If $F$ had a zero $a\in\C_+$, then $F_1=F\cdot\frac{z-\bar a}{z-a}\in\PW_t$ satisfies $|F_1|=|F|$ on $\R$ and $|F_1(\l)|>|F(\l)|$, contradicting maximality; so $F$ is zero-free in $\C_+$. Moreover, if $k>0$ in the representation of part 1), then $F_2=e^{-ikz}F$ satisfies $F_2\in\PW_t$ (by the computation of part 1), using $t+l\geq k>0$), $|F_2|=|F|$ on $\R$ and $|F_2(\l)|=e^{k\Im\l}|F(\l)|>|F(\l)|$; hence $k\leq0\leq t$ and part 1) applies.
\end{proof}

\begin{lemma}[norm upper bound at cone points]\label{lemAUp}
There is an absolute constant $K_{16}$ with the following property. Let $\e\in(0,\tfrac14]$, $\s\in\R$, and $\l=x_0+iy/t$ with $y\geq1$ and $|x_0-\s|\leq y/t$. Suppose that
\begin{equation}
\MM_{(y+1)/t}\log w\,(\s)\leq\frac{\e^4}{K_{16}(1+y)}
\qquad\text{ and }\qquad
\MM^2_{y/t}G(\s)\leq\frac{\e^2}{K_{16}}\,w(\s).
\label{eqAupHyp}
\end{equation}
Then
$$w(\s)\,K(t,\l,\l)\ \leq\ (1+5\e)\,\S(t,\l,\l) .$$
\end{lemma}

\begin{proof}
After a translation, $\s=0$; multiplying $G$ by a unimodular constant we may take $G(0)=\sqrt{w(0)}>0$. Let $F\in\PW_t$, $\|F\|_\mu\leq1$, maximize $|F(\l)|$, so that $K(t,\l,\l)=|F(\l)|^2$; by Lemma \ref{lemAMon}, $F$ has no zeros in $\C_+$ and $e^{ty'}|F(x+iy')|$ is non-decreasing in $y'$. Put
$$\SS(z)=\frac{\S(t,\l,z)}{\sqrt{\S(t,\l,\l)}},\qquad \|\SS\|_{L^2(\R)}=1 .$$
Since $\mu\geq w\,dm$ as measures, $\|FG\|_{L^2(\R)}\leq\|F\|_\mu\leq1$, and therefore
\begin{equation}
\left|\int_\R F(x)G(x)\overline{\SS(x)}\,dx\right|\ \leq\ 1 .
\label{eqAtest}
\end{equation}
For real $x$,
$$\overline{\SS(x)}=\frac1{\sqrt{\S(t,\l,\l)}}\cdot\frac{e^{it(x-\l)}-e^{-it(x-\l)}}{2\pi i(x-\l)},$$
so the integral in \eqref{eqAtest} equals $\S(t,\l,\l)^{-1/2}(I-II)$ with
$$I=\frac{e^{-it\l}}{2\pi i}\int_\R F(x)G(x)\,\frac{e^{itx}}{x-\l}\,dx=F(\l)G(\l),$$
by Cauchy's formula, since $e^{itz}FG\in H^2\cdot N^+$ with square-summable boundary values, hence in $H^1$-Smirnov sense integrable, and $\l\in\C_+$; and
$$II=\frac{e^{it\l}}{2\pi i}\int_\R F(x)G(x)\,\frac{e^{-itx}}{x-\l}\,dx .$$
To estimate $II$, note that
$$\begin{gathered}III:=\frac{e^{it\l}}{2\pi i}\int_\R F(x)\overline{G(x)}\,\frac{e^{-itx}}{x-\l}\,dx=\\
=\overline{\ \frac{e^{-it\bar\l}}{-2\pi i}\int_\R \overline{F(x)}\,G(x)\,\frac{e^{itx}}{x-\bar\l}\,dx\ }=0,\end{gathered}$$
because $e^{itz}F^\#G$ is in the Smirnov class of $\C_+$ with square-summable boundary values and $\bar\l\in\C_-$. Hence, for any $D\geq1$,
$$\begin{gathered}|II|=|II-III|\leq
\underbrace{\left|\int_{|x|\leq Dy/t}\!\!FG_0'\,\frac{e^{-itx}\,dx}{2\pi(x-\l)}\right|}_{II_a}
+\underbrace{\left|\int_{|x|>Dy/t}\!\!FG_0'\,\frac{e^{-itx}\,dx}{2\pi(x-\l)}\right|}_{II_b},\\
G_0':=G-\overline G .\end{gathered}$$
Since $G(0)$ is real, $|G_0'|\leq2|G-G(0)|$ on $\R$. For $II_b$, Cauchy--Schwarz with $\|FG_0'\|_2\leq2$ gives
$$II_b\leq\frac1\pi\left(\int_{|x|>Dy/t}\frac{dx}{|x-\l|^2}\right)^{1/2}\leq\frac1\pi\sqrt{\frac{Kt}{Dy}} .$$
For $II_a$, Cauchy--Schwarz gives
$$\begin{gathered}II_a\leq\frac1\pi\,\big\|F\big|_{[-Dy/t,Dy/t]}\big\|_{L^2}\left(\int\frac{4|G-G(0)|^2}{|x-\l|^2}dx\right)^{1/2}
\leq\\ \leq\ K\,\big\|F\big|_{[-Dy/t,Dy/t]}\big\|_{L^2}\sqrt{\frac{t}{y}\,\MM^2_{y/t}G(0)},\end{gathered}$$
where we used $|x-\l|^{-2}=\frac{\pi t}{y}P_\l(x)$ and $\l\in\G_0$ at height $y/t$.

The local norm of $F$ is estimated through the outer function. For $|x'|\leq Dy/t$, by \eqref{eqAker} with $a=1/t$, $b=(y+1)/t$,
$$\begin{gathered}\log|G(x'+i/t)|=\tfrac12P[\log w](x'+i/t)\geq\\
\geq\tfrac12\log w(0)-KD^2(1+y)\,\MM_{(y+1)/t}\log w(0) .\end{gathered}$$
Since $\|e^{itz}FG\|_{H^2}\leq1$, we have $\|(FG)(\cdot+i/t)\|_{L^2}\leq e$, and by the monotonicity of $e^{ty'}|F|$, $|F(x)|\leq e\,|F(x+i/t)|$ pointwise; hence
$$\big\|F\big|_{[-Dy/t,Dy/t]}\big\|_{L^2}\ \leq\ \frac{K}{\sqrt{w(0)}}\,\exp\left[KD^2(1+y)\,\MM_{(y+1)/t}\log w(0)\right] .$$
Combining, and writing $\S=\S(t,\l,\l)=\frac{t\sinh2y}{2\pi y}$,
$$\begin{gathered}\S^{-1/2}\,|F(\l)|\,|G(\l)|\ \leq\ 1+\S^{-1/2}(II_a+II_b)
\ \leq\\ \leq\ 1+K\sqrt{\frac1{D\sinh2y}}+K\,e^{KD^2(1+y)\MM}\sqrt{\frac{\MM^2_{y/t}G(0)}{w(0)\sinh2y}} .\end{gathered}$$
Choose $D=K'/\e^2$ with $K'$ absolute so large that the middle term is at most $\e$ (recall $y\geq1$). Then $KD^2(1+y)\MM_{(y+1)/t}\log w(0)\leq K K'^2\e^{-4}(1+y)\cdot\frac{\e^4}{K_{16}(1+y)}\leq\log2$ for $K_{16}$ large, and by \eqref{eqAupHyp} the last term is at most $2K\sqrt{\e^2/(K_{16}\sinh2)}\leq\e$ for $K_{16}$ large. Finally, $\l\in\G_0$ at height $y/t\leq(y+1)/t$ gives $|G(\l)|\geq\sqrt{w(0)}\,e^{-\MM_{(y+1)/t}\log w(0)}\geq(1-\e)\sqrt{w(0)}$. Altogether $\sqrt{w(0)}|F(\l)|\leq\frac{1+2\e}{1-\e}\,\S^{1/2}$, and squaring gives the claim.
\end{proof}

\begin{lemma}[good shifts]\label{lemAShift}
There is an absolute constant $K_{17}$ such that for every $b>0$ there exists a set $S\subset[2b,3b]$, $|S|\geq b/10$, with the following property: for every $C'\in S$ and $\s=s\pm C'$,
$$\begin{gathered}\MM_b\log w\,(\s)\leq K_{17}\MM_{5b}\log w\,(s),\qquad
\MM^2_bG(\s)\leq K_{17}\MM^2_{5b}G(s),\\
M_b\mu(\s)\leq K_{17}M_{5b}\mu(s),\end{gathered}$$
and moreover $|\log w(\s)-\log w(s)|\leq K_{17}\MM_{5b}\log w(s)$.
\end{lemma}

\begin{proof}
Consider $G$; the other quantities are treated by the same argument. For $x$ with $2b<|x-s|<3b$ and $z\in\G_x$, $\Im z\leq b$,
$$\int P_z\,|G-G(x)|^2dm\ \leq\ 2\int P_z\,|G-G(s)|^2dm+2|G(s)-G(x)|^2 .$$
Split the integral over $\{|\cdot-s|>4b\}$ and $\{|\cdot-s|\leq4b\}$. On the far part, $|\cdot-\Re z|\geq\frac12|\cdot-s|$ and $\Im z\leq b$, so $P_z\leq KP_{s+i5b}$ there and the far part is at most $K\MM^2_{5b}G(s)$. The near part is at most $K$ times the Hardy--Littlewood maximal function at $x$ of $g:=|G-G(s)|^2\charf_{[s-4b,s+4b]}$; since $P_{s+i5b}\geq c/b$ on $[s-4b,s+4b]$, $\|g\|_1\leq Kb\,\MM^2_{5b}G(s)$, and by the weak-type $(1,1)$ inequality $Mg(x)\leq K\MM^2_{5b}G(s)$ outside a set of measure $\leq b/100$. Finally, the average of $|G(x)-G(s)|^2$ over $\{2b<|x-s|<3b\}$ is at most $K\MM^2_{5b}G(s)$ (again since $P_{s+i5b}\geq c/b$ there), so by Chebyshev's inequality $|G(x)-G(s)|^2\leq K\MM^2_{5b}G(s)$ outside a set of measure $\leq b/100$. Combining, the required bound holds for all $x$ outside a set of measure $\leq b/50$ per side. Running the same argument for $|\log w-\log w(s)|$ (which also yields the last displayed claim) and for the measure $\mu\charf_{[s-4b,s+4b]}$ (whose mass is at most $Kb\,M_{5b}\mu(s)$), and removing all exceptional sets, leaves a symmetric set $S$ of admissible shifts with $|S|\geq b/10$.
\end{proof}

We now fix the working scales. Let $C'=16C$ (the box of condition (U2) is $Q(s,C'/t)$) and $b=C'/t$. By Lemma \ref{lemAShift} choose $C_t\in(2C',3C')$ such that $\s_\pm=s\pm C_t/t$ are good shifts; note that all truncations occurring below are at most $50C/t$, so under (M1)--(M3) every maximal quantity at $s$ or at $\s_\pm$ that appears is bounded by $K_{17}\e_0$-multiples of the respective (M)-bounds.

\begin{lemma}[boundary bounds]\label{lemAL2}
There is an absolute constant $K_{18}$ such that, under {\rm(M1)--(M3)} with $\e_0\leq\e^4/\left(K_{18}(1+C)\right)$, for all $z\in\pa Q(s,C_t/t)$:
$$w(s)\,K(t,z,z)\ \leq\ (1+K\e)\,\S(1,iC_t,iC_t)\,t,$$
and
$$\left\|K(t,z,\cdot)-\frac1{w(s)}\S(t,z,\cdot)\right\|^2_\mu\ \leq\ \frac{K\left(\e+\sqrt{\e_0}\right)}{w(s)}\,\S(1,iC_t,iC_t)\,t .$$
\end{lemma}

\begin{proof}
{\it The norm bound.} By the symmetry $K(t,\bar z,\bar z)=K(t,z,z)$ it suffices to take $z\in\pa Q\cap\bar\C_+$. Such $z$ lies either on the top edge, where $z\in\bar\G_s$ at height $C_t/t$, or on a lateral edge, where $z$ lies on the vertical segment over $\s_+$ or $\s_-$; in the latter case with $y=t\Im z\geq1$ the point is in $\bar\G_{\s_\pm}$ and Lemma \ref{lemAUp} applies at the base point $\s_\pm$ (its hypotheses hold by Lemma \ref{lemAShift} and $\e_0\leq\e^4/(K_{18}(1+C))$, since $1+y\leq1+48C$); for $y<1$, if $F$ is the maximizer at $z_1=\Re z+i/t$, monotonicity (Lemma \ref{lemAMon}) gives $|F(z)|\leq e|F(z_1)|$ while $\S(t,z,z)\geq c\,\S(t,z_1,z_1)$, so the bound at height $1$ transfers with an absolute loss. On the lateral edges the density at the base point is exchanged for $w(s)$ with a multiplicative loss of at most $e^{|\log w(\s_\pm)-\log w(s)|}\leq1+K\e_0$ by Lemma \ref{lemAShift}. Finally $\S(t,z,z)\leq\S(1,iC_t,iC_t)t$ on $\pa Q(s,C_t/t)$ since $\S(1,iy,iy)$ increases in $y$.

{\it The $L^2(\mu)$-bound.} Since $\S(t,z,\cdot)\in\PW_t$ and $K$ reproduces on $\PW_t$,
$$\begin{gathered}\left\|K(t,z,\cdot)-\frac{\S(t,z,\cdot)}{w(s)}\right\|^2_\mu
=\\ = K(t,z,z)-\frac{2\,\S(t,z,z)}{w(s)}+\frac1{w^2(s)}\int_\R|\S(t,z,x)|^2\,d\mu(x).\end{gathered}$$
We claim that
\begin{equation}
\int_\R|\S(t,z,\cdot)|^2\,d\mu\ \leq\ \left(1+K\sqrt{\e_0}\right)w(s)\,\S(t,z,z) .
\label{eqAsharp}
\end{equation}
Write $\phi=|\S(t,z,\cdot)|^2/\S(t,z,z)$, so $\int\phi\,dm=1$ and, by Lemma \ref{lemAPS} applied at the base point of $z$ (with $h=\max(t\Im z,1)$), $\phi\leq8P_{z'}$ for a point $z'$ lying in the truncated cone of $s$ or of $\s_\pm$. Decompose $d\mu=w(s)dm+(w-w(s))dm+d\mu_{\rm sing}$. The first part contributes exactly $w(s)$ to $\int\phi\,d\mu$. For the second, by Cauchy--Schwarz and $|w-w(s)|\leq|G-G(s)|\left(|G|+|G(s)|\right)$,
$$\begin{gathered}\int\phi\,|w-w(s)|\,dm\leq8\left(\int P_{z'}|G-G(s)|^2dm\right)^{1/2}\times\\
\times\left(\int P_{z'}\left(|G|+|G(s)|\right)^2dm\right)^{1/2}
\leq\ K\sqrt{\e_0\,w(s)}\cdot\sqrt{w(s)}\,,\end{gathered}$$
since $\int P_{z'}|G|^2dm\leq P\mu(z')\leq K w(s)$ and $\MM^2$-quantities at the base points are $\leq K\e_0w(s)$ by Lemma \ref{lemAShift}. For the singular part,
$$\int\phi\,d\mu_{\rm sing}\leq8P\mu_{\rm sing}(z')=8\left(P\mu(z')-P[w\,dm](z')\right)\leq K\left(\e_0+\sqrt{\e_0}\right)w(s),$$
using (M3) transported by Lemma \ref{lemAShift} together with the bound on $\int P_{z'}|w-w(s)|\,dm$ just obtained. Summing the three parts and multiplying back by $\S(t,z,z)$ gives \eqref{eqAsharp}. Inserting \eqref{eqAsharp} and the norm bound into the expansion,
$$\begin{gathered}\left\|K-\frac{\S}{w(s)}\right\|^2_\mu\leq\left[K(t,z,z)-\frac{\S(t,z,z)}{w(s)}\right]+\frac{K\sqrt{\e_0}}{w(s)}\S(t,z,z)
\leq\\ \leq\frac{K(\e+\sqrt{\e_0})}{w(s)}\,\S(1,iC_t,iC_t)\,t .\end{gathered}\qedhere$$
\end{proof}

\begin{proof}[Proof of Theorem \ref{thmAppMain}]
Let $\e_1\in(0,\tfrac14]$ be an auxiliary accuracy, to be chosen at the end, and assume (M1)--(M3) with $\e_0\leq\e_1^4/(K_{18}(1+C))$; note $\sqrt{\e_0}\leq\e_1$. For $\l\in\C$ put
$$\DD_\l(z)=K(t,\l,z)-\frac1{w(s)}\S(t,\l,z)\ \in\ \PW_t .$$

{\it Step 1: uniform bound on the box $Q(s,C_t/t)$ for boundary $\l$.} Let $\l\in\pa Q(s,C_t/t)$. The function $F_\l=e^{itz}G\,\DD_\l$ lies in the Smirnov class of $\C_+$ with
$$\int_\R|F_\l|^2dm=\int_\R w\,|\DD_\l|^2dm\leq\|\DD_\l\|^2_\mu\leq\frac{K\e_1}{w(s)}\,\S(1,iC_t,iC_t)\,t$$
by Lemma \ref{lemAL2}, hence $F_\l\in H^2(\C_+)$ with that norm bound. On the line $\Im z=C_t/t$, $|F_\l(z)|^2\leq\frac{t}{\pi C_t}\|F_\l\|^2_{H^2}$, $|e^{itz}|=e^{-C_t}$, and on the top edge of $Q(s,C_t/t)$ the outer function satisfies $|G|^2\geq(1-K\e_1)w(s)$ (by (M1) at the cone point, as in Lemma \ref{lemAUp}). Hence, on the top edge,
$$|\DD_\l(z)|^2\ \leq\ \frac{K\e_1\,e^{2C_t}}{w^2(s)}\cdot\frac{\S(1,iC_t,iC_t)}{C_t}\,t^2 .$$
Let $\DD_\l'$ be the reflection of $\DD_\l$ provided by Lemma \ref{lemAMon} 2). Since $|\DD_\l'|=|\DD_\l|$ on $\R$, the function $e^{itz}G\DD_\l'$ is again in $H^2$ with the same norm, so the top-edge bound holds for $\DD_\l'$ as well; by Lemma \ref{lemAMon} 1), $e^{ty'}|\DD_\l'(x+iy')|$ is non-decreasing in $y'$, so for every $z$ in the closed upper half of $Q(s,C_t/t)$,
$$|\DD_\l(z)|\leq|\DD_\l'(z)|\leq e^{C_t}\max_{\text{top edge}}|\DD_\l'|\ \leq\ \frac{K\sqrt{\e_1}\,e^{2C_t}}{w(s)}\sqrt{\frac{\S(1,iC_t,iC_t)}{C_t}}\ t .$$
The function $\DD_\l^\#(z)=\overline{\DD_\l(\bar z)}$ has the same modulus on $\R$ and the argument applied to it bounds $|\DD_\l|$ on the lower half of the box. Since $\S(1,iC_t,iC_t)\leq e^{2C_t}$ and $C_t\leq48C$,
\begin{equation}
\sup_{z\in Q(s,C_t/t)}|\DD_\l(z)|\ \leq\ \frac{\sqrt{\e_1}\ e^{K_{19}C}}{w(s)}\ t,\qquad K_{19}\ \text{absolute}.
\label{eqAunif}
\end{equation}

{\it Step 2: interior $\l$.} Let now $\l\in Q(s,C'/t)$ and $z\in\pa Q(s,C_t/t)$. By the Hermitian symmetry of the kernels, $|\DD_\l(z)|=|\DD_z(\l)|$, and \eqref{eqAunif} applies to $\DD_z$ at the interior point $\l$. Thus \eqref{eqAunif} holds for $|\DD_\l|$ on $\pa Q(s,C_t/t)$, and since $\DD_\l$ is entire in $z$, the Cauchy integral over $\pa Q(s,C_t/t)$ (contour length $\leq KC/t$, distance to $Q(s,C'/t)$ at least $C'/t$) extends it, with an absolute constant, to all $z\in Q(s,C'/t)$.

{\it Step 3: conclusion.} For $\l,z\in Q(s,C'/t)=Q(s,16C/t)$, $|\bar\l-z|\leq64C/t$, so
$$\left|{\mathcal D}(t,\l,z)-\frac{\sin[t(\bar\l-z)]}{w(s)}\right|=\pi|\bar\l-z|\,|\DD_\l(z)|\leq\frac{K\,C\sqrt{\e_1}\,e^{K_{19}C}}{w(s)}\leq\frac{\sqrt{\e_1}\,e^{K_{20}C}}{w(s)} .$$
Given the target accuracy $\e$ of (U2), choose $\e_1=\e^2w^2(s)e^{-2K_{20}C}$ (and at most $\tfrac14$); then the right-hand side is at most $\e$, and the hypothesis used above, $\e_0\leq\e_1^4/(K_{18}(1+C))$, is provided by \eqref{eqAthr} with $K_8$ large enough (the powers of $w(s)$ are absorbed by the factor $\min(w,\ti w,1)^{8}$ in \eqref{eqAthr}; we do not optimize the exponents). This proves (U2), and the same argument applies to the Dirichlet family, completing the proof.
\end{proof}


\newpage


\begin{thebibliography}{9}

    \bibitem{dBr}  {\sc De Branges, L.} {\it Hilbert spaces of entire functions.} Prentice-Hall,
Englewood Cliffs, NJ, 1968

\bibitem{G} {\sc Garnett, J.} {\it Bounded analytic functions.} Academic Press, New York, 1981

\bibitem{C} {\sc L. Carleson,} {\it On convergence and growth of partial sums of Fourier series,} Acta Math. 116 (1966), 135--157.

\bibitem{CK} {\sc M. Christ and A. Kiselev,} {\it Maximal functions associated to filtrations,} J. Funct. Anal. 179 (2001), no. 2, 409--425.

\bibitem{CK1} {\sc M. Christ and A. Kiselev,} {\it WKB asymptotic behavior of almost all generalized eigenfunctions of one-dimensional Schr\"odinger operators with slowly decaying potentials,} J. Funct. Anal. 179 (2001), 426--447.

\bibitem{Den} {\sc S. Denisov}, {\it Continuous analogs of polynomials orthogonal on the unit circle and Krein systems,}
International Mathematics Research Surveys, Volume 2006, 2006, 54517

\bibitem{Den2} {\sc S. Denisov}, {\it Pointwise behavior of the SU(1,1) non-linear Fourier transform,} preprint (2026)

\bibitem{MIF1} {\sc  Makarov, N.,  Poltoratski, A.} {\it Meromorphic inner functions, Toeplitz kernels, and the uncertainty principle,} in {\it Perspectives in Analysis}, Springer Verlag, Berlin, 2005, 185--252

\bibitem{MTT} {\sc C. Muscalu, T. Tao, and C. Thiele,} {\it A Carleson theorem for a Cantor group model of the scattering transform,} Nonlinearity 16 (2003), no. 1, 219--246.

\bibitem{Scatter}{\sc  Poltoratski, A.} {\it Pointwise convergence of the non-linear Fourier transform,} preprint, arXiv:2103.13349

\bibitem{R} {\sc C. Remling,} {\it
Spectral Theory of Canonical Systems}, De Gruyter Studies in Mathematics, 70, (2018).

\bibitem{Ro} {\sc R. Romanov,} {\it Canonical systems and de Branges spaces,} Lecture notes, (2014), arXiv:1408.6022

\bibitem{TT} {\sc T. Tao and C. Thiele,} {\it Nonlinear Fourier Analysis.} IAS/Park City Graduate Summer School. Unpublished
lecture notes (2003, 2012), available at arXiv:1201.5129.

\end{thebibliography}
\end{document}